\documentclass[11pt]{amsart}
\usepackage{amsmath, amssymb}
\usepackage{amsfonts}
\usepackage{graphicx}
\usepackage{mathrsfs}
\usepackage[arrow,matrix,curve,cmtip,ps]{xy}
\usepackage{amsthm}
\usepackage{enumerate}
\usepackage{color}
\usepackage[colorlinks]{hyperref}
\usepackage{tikz}

\allowdisplaybreaks

\newtheorem{thm}{Theorem}[section]
\newtheorem{lem}[thm]{Lemma}
\newtheorem{prop}[thm]{Proposition}
\newtheorem{cor}[thm]{Corollary}
\newtheorem*{theorem*}{Theorem}
\theoremstyle{remark}
\newtheorem{rem}[thm]{Remark}

\newtheorem{defn}[thm]{Definition}
\newtheorem{ex}[thm]{Example}

\numberwithin{equation}{section}

\newcommand{\g}{\mathcal{G}}
\newcommand{\go}{\mathcal{G}^0}
\newcommand{\gl}{\mathcal{G}^1}

\newcommand{\pGo}{\Phi(G^0)}
\newcommand{\gh}{\mathcal{G}/(H,B)}
\newcommand{\pgo}{\Phi(\mathcal{G}^0)}
\newcommand{\pgl}{\Phi(\mathcal{G}^1)}
\newcommand{\sg}{\Phi_{\mathrm{sg}}(G^0)}

\begin{document}
\title[Quotients of Ultragraph $C^*$-Algebras]{Quotients of Ultragraph $C^*$-Algebras}

\author[Hossein Larki]{Hossein Larki}

\address{Department of Mathematics\\
Faculty of Mathematical Sciences and Computer\\
Shahid Chamran University of Ahvaz\\
 Iran}
\email{h.larki@scu.ac.ir}


\date{\today}



\begin{abstract}
Let $\g$ be an ultragraph and let $C^*(\g)$ be the associated $C^*$-algebra introduced by Mark Tomforde. For any gauge invariant ideal $I_{(H,B)}$ of $C^*(\g)$, we analyze the structure of the quotient $C^*$-algebra $C^*(\g)/I_{(H,B)}$. For simplicity's sake, we first introduce the notion of quotient ultragraph $\gh$ and an associated $C^*$-algebra $C^*(\gh)$ such that $C^*(\gh)\cong C^*(\g)/I_{(H,B)}$. We then prove the gauge invariant and the Cuntz-Krieger uniqueness theorems for $C^*(\gh)$ and describe primitive gauge invariant ideals of $C^*(\gh)$.
\end{abstract}

\maketitle

\section{Introduction}

In order to bring graph $C^*$-algebras and Exel-Laca algebras together under one theory, Tomforde introduced in \cite{tom} the notion of ultragraphs and associated $C^*$-algebras. An ultragraph is basically a directed graph in which the range of each edge is allowed to be a nonempty set of vertices rather than a single vertex. However, the class of ultragraph $C^*$-algebras are strictly lager than the graph $C^*$-algebras as well as the Exel-Laca algebras (see \cite[Section 5]{tom2}). Due to some similarities, parts of fundamental results for graph $C^*$-algebras, such as the Cuntz-Krieger and the gauge invariant uniqueness theorems, simplicity, and $K$-theory computation have been extended to the setting of ultragraphs \cite{tom,tom2}.

Furthermore, by constructing a specific topological quiver $\mathcal{Q}(\g)$ from an ultragraph $\g$, Katsura et al. described some properties of the ultragraph $C^*$-algebra $C^*(\g)$ using those of topological quivers \cite{kat3}. In particular, they showed that every gauge invariant ideal of  $C^*(\g)$ is of the form $I_{(H,B)}$ corresponding to an admissible pair $(H,B)$ in $\g$. Recall that the graph $C^*$-algebras (and also the ultragraph $C^*$-algebras) are simple examples of the topological quiver algebras \cite{muh} in which the sets of vertices and edges are considered with discrete topology and Radon measures are special counting measures. So, it seems that working with topological quiver algebras is more complicated than with the ultragraph $C^*$-algebras.

Recall that for any gauge invariant ideal $I_{(H,B)}$ of a graph $C^*$-algebra $C^*(E)$, there is a (quotient) graph $E/(H,B)$ such that $C^*(E)/I_{(H,B)}\cong C^*(E/(H,B))$ (see \cite {bat2,bat}). So, the class of graph $C^*$-algebras contains such quotients, and we may apply the results and properties of graph $C^*$-algebras for their quotients. For examples, some contexts such as simplicity, $K$-theory, primitivity, and topological stable rank are directly related to the structure of ideals and quotients.

Unlike the $C^*$-algebras of graphs and topological quivers, the class of ultragraph $C^*$-algebras is not closed under quotients. This causes some obstacles in studying the structure of ultragraph $C^*$-algebras. The aims of this article are to prove two kinds of uniqueness theorems, known as the gauge invariant and the Cuntz-Krieger uniqueness theorems, for quotients of ultragraph $C^*$-algebras and apply them for analyzing the ideal structure of an ultragraph $C^*$-algebra. Suppose that $I_{(H,B)}$ is a gauge invariant ideal of an ultragraph $C^*$-algebra $C^*(\g)$. For the sake of convenience, we first introduce the notion of quotient ultragraph $\gh$ and a relative $C^*$-algebra $C^*(\gh)$ such that $C^*(\g)/I_{(H,B)}\cong C^*(\gh)$ and then prove the gauge invariant and Cuntz-Krieger uniqueness theorems for $C^*(\gh)$. The $C^*$-algebra $C^*(\gh)$ is called a {\it quotient ultragraph $C^*$-algebra}. The uniqueness theorems help us to show when a representation of $C^*(\gh)$ is injective. We see that the structure of $C^*(\gh)$ is close to that of graph and ultragraph $C^*$-algebras. Hence, many traditional graph $C^*$-algebraic techniques may be extended to the ultragraph setting by applying quotient ultragraphs (see Sections 6 and 7). We note that the initial idea for defining quotient ultragraphs has been inspired from quotient labelled graphs in \cite{jeo}. However, many main results of this paper are not known for labelled graphs. Moreover, any restrictive condition such as set-finiteness or receiver set-finiteness does not be assumed here for ultragraphs.

This article is organized as follows. We begin in Section 2 by giving some definitions and preliminaries about the ultragraphs and their $C^*$-algebras which will be used in the next sections. In Section 3, for any admissible pair $(H,B)$ in an ultragraph $\g$, we introduce the quotient ultragraph $\gh$ and an associated $C^*$-algebra $C^*(\gh)$. For this, the ultragraph $\g$ is modified by an extended ultragraph $\overline{\g}$ and we define an equivalent relation $\sim$ on $\overline{\g}$. Then the quotient ultragraph $\gh$ is ultragraph $\overline{\g}$ with the equivalent classes $\{[A]: A\in \overline{\g}^0\}$. In Section 4, by approaching with graph $C^*$-algebras, the gauge invariant and the Cuntz-Krieger uniqueness theorems will be proved for the quotient ultragraphs $C^*$-algebras.

In Section 5, we describe the gauge invariant ideals of $C^*(\gh)$. We first prove that $C^*(\gh)$ is isometrically isomorphic to the quotient $C^*$-algebra $C^*(\g)/I_{(H,B)}$. We then see that every gauge invariant ideal $I_{(K,S)}$ in $C^*(\g)$ with $K\supseteq H$ and $K\cup S \supseteq B$ induces an ideal $J_{[K,S]}$ in $C^*(\gh)$ and all gauge invariant ideals of $C^*(\gh)$ are of this form. In Sections 6 and 7, using quotient ultragraphs, we can apply some graph $C^*$-algebra methods for the ultragraph $C^*$-algebras. In Section 6, we define Condition (K) for $\gh$ under which all ideals of $C^*(\gh)$ are gauge invariant and the real rank of $C^*(\gh)$ is zero. As a corollary, we can recover both \cite[Proposition 7.3]{kat3} and \cite[Proposition 5.26]{kat2} by a different approach. Finally, in Section 7, all primitive gauge invariant ideals of ultragraph $C^*$-algebras will be characterized.


\section{preliminaries}

In this section, we review basic definitions and properties of ultragraph $C^*$-algebras which will be needed through the paper. For more details, we refer the reader to \cite{tom} and \cite{kat3}.

\begin{defn}[\cite{tom}]
An \emph{ultragraph} is a quadruple $\g=(G^0,\gl,r_\g,s_\g)$ consisting of a countable vertex set $G^0$, a countable edge set $\gl$, the source map $s_\g: \gl \rightarrow G^0$, and the range map $r_\g:\gl \rightarrow \mathcal{P}(G^0)\setminus \{\emptyset\}$, where $\mathcal{P}(G^0)$ is the collection of all subsets of $G^0$. If $r_\g(e)$ is a singleton vertex for each edge $e\in \gl$, then $\g$ is an ordinary (directed) graph.
\end{defn}

For our convenience, we use the notation $\go$ of \cite{kat3} rather than \cite{tom,tom2}. For any set $X$, a nonempty subcollection of the power set $\mathcal{P}(X)$ is said to be an \emph{algebra} if it is closed under the set operations $\cap$, $\cup$, and $\setminus$. If $\g$ is an ultragraph, the smallest algebra in $\mathcal{P}(G^0)$ containing $\{\{v\}:v\in G^0\}$ and $\{r_\g(e):e\in \gl\}$ is denoted by $\go$. We simply denote every singleton set $\{v\}$ by $v$. So, $G^0$ may be considered as a subset of $\go$.

\begin{defn}
For each $n\geq 1$, a \emph{path $\alpha$ of length $|\alpha|=n$} in $\g$ is a sequence $\alpha=e_1\ldots e_n$ of edges such that $s(e_{i+1})\in r(e_i)$ for $1\leq i\leq n-1$. If also $s(e_1)\in r(e_n)$, $\alpha$ is called a \emph{loop} or a \emph{closed path}. We write $\alpha^0$ for the set $\{s_\g(e_i): 1\leq i\leq n\}$. The elements of $\go$ are considered as the paths of length zero. The set of all paths in $\g$ is denoted by $\g^*$. We may naturally extend the maps $s_\g,r_\g$ on $\g^*$ by defining $s_\g(A)=r_\g(A)=A$ for $A\in \go$, and $r_\g(\alpha)=r_\g(e_n)$, $s_\g(\alpha)=s_\g(e_1)$ for each path $\alpha=e_1\ldots e_n$.
\end{defn}

\begin{defn}[\cite{tom}]\label{defn2.3}
Let $\g$ be an ultragraph. A \emph{Cuntz-Krieger $\g$-family} is a set of partial isometries $\{s_e:e\in\gl\}$ with mutually orthogonal ranges and a set of projections $\{p_A:A\in\go\}$ satisfying the following relations:
\begin{enumerate}
\item $p_\emptyset=0$, $p_A p_B=p_{A\cap B}$, and $p_{A\cup B}=p_A +p_B-p_{A\cap B}$ for all $A,B\in\go$,
\item $s_e^* s_e=p_{r_\g(e)}$ for $e\in\gl$,
\item $s_e s_e^*\leq p_{s_\g(e)}$ for $e\in \gl$, and
\item $p_v=\sum_{s_\g(e)=v}s_e s_e^*$ whenever $0<|s_\g^{-1}(v)|<\infty$.
\end{enumerate}
The $C^*$-algebra $C^*(\g)$ of $\g$ is the (unique) $C^*$-algebra generated by a universal Cuntz-Krieger $\g$-family.
\end{defn}

By \cite[Remark 2.13]{tom}, we have
$$C^*(\g)=\overline{\mathrm{span}}\left\{s_\alpha p_A s_\beta^*:\alpha,\beta\in \g^*, A\in \go, \mathrm{~and~} r_\g(\alpha)\cap r_\g(\beta)\cap A\neq \emptyset \right\},$$
where $s_\alpha:=s_{e_1}\ldots s_{e_n}$ if $\alpha=e_1\ldots e_n$, and $s_\alpha:=p_A$ if $\alpha=A$.

\begin{rem}\label{rem2.4}
As noted in \cite[Section 3]{tom}, every graph $C^*$-algebra is an ultragraph $C^*$-algebra. Recall that if $E=(E^0,E^1,r_E,s_E)$ is a directed graph, a collection $\{s_e,p_v:v\in E^0, e\in E^1\}$ containing mutually orthogonal projections $p_v$ and partial isometries $s_e$ is called a Cuntz-Krieger $E$-family if
\begin{enumerate}
\item $s_e^* s_e=p_{r_E(e)}$ for all $e\in E^1$,
\item $s_e s_e^*\leq p_{s_E(e)}$ for all $e\in E^1$, and
\item $p_v=\sum_{s_E(e)=v}s_e s_e^*$ for every vertex $v\in E^0$ with $0<|s_E^{-1}(v)|<\infty$.
\end{enumerate}
We denote by $C^*(E)$ the universal $C^*$-algebra generated by a Cuntz-Krieger $E$-family.
\end{rem}

By the universal property, $C^*(\g)$ admits the \emph{gauge action} of the unit circle $\mathbb{T}$. By an \emph{ideal}, we mean a closed two-sided ideal. Using the properties of quiver $C^*$-algebras \cite{kat3}, the gauge invariant ideals of $C^*(\g)$ were characterized in \cite[Theorem 6.12]{kat3} via a one-to-one correspondence with the admissible pairs of $\g$.

\begin{defn}\label{defn2.5}
A subset $H\subseteq \go$ is said to be \emph{hereditary} if the following properties holds:
\begin{enumerate}
\item $s_\g(e)\in H$ implies $r_\g(e)\in H$ for all $e\in \gl$.
\item $A\cup B\in H$ for all $A,B\in H$.
\item If $A\in H,~B\in\go$, and $B\subseteq A$, then $B\in H$.
\end{enumerate}
Moreover, a subset $H\subseteq \go$ is called \emph{saturated} if for any $v\in G^0$ with $0<|s_\g^{-1}(v)|<\infty$, then $r_\g(s_\g^{-1}(v))\subseteq H$ implies $v\in H$. The \emph{saturated hereditary closure} of a subset $H\subseteq \go$ is the smallest hereditary and saturated subset $\overline{H}$ of $\go$ containing $H$.
\end{defn}

Let $H$ be a saturated hereditary subset of $\go$. The set of \emph{breaking vertices of $H$} is denoted by
$$B_H:=\left\{ w\in G^0: |s_\g^{-1}(w)|=\infty ~~\mathrm{but}~~ 0<|r_\g(s_\g^{-1}(w))\cap (\go\setminus H)|<\infty\right\}.$$
Following \cite{kat2}, an \emph{admissible pair $(H,B)$ in $\g$} is a saturated hereditary set $H\subseteq \go$ together with a subset $B\subseteq B_H$. For any admissible pair $(H,B)$ in $\g$, we define the ideal $I_{(H,B)}$ of $C^*(\g)$ generated by
$$\{p_A:A\in \go\} \cup \left\{p_w^H:w\in B\right\},$$
where $p_w^H:=p_w-\sum_{s_\g(e)=w,~ r_\g(e)\notin H}s_e s_e^*$. Note that the ideal $I_{(H,B)}$ is gauge invariant and \cite[Theoerm 6.12]{kat1} implies that every gauge invariant ideal $I$ of $C^*(\g)$ is of the form $I_{(H,B)}$ for
$$H:=\left\{A:p_A\in I\right\} ~~ \mathrm{and} ~~ B:=\left\{ w\in B_H: p_w^H\in I\right\}.$$


\section{Quotient Ultragraphs and their $C^*$-algebras}

In this section, for any admissible pair $(H,B)$ in an ultragraph $\g$, we introduce the quotient ultragraph $\g/(H,B)$ and its relative $C^*$-algebra $C^*(\g/(H,B))$. We will show in Proposition \ref{prop5.1} that $C^*(\g/(H,B))$ is isometrically isomorphic to the quotient $C^*$-algebra $C^*(\g)/ I_{(H,B)}$.

Let us fix an ultragraph $\g=(G^0,\g^0,r_\g,s_\g)$ and an admissible pair $(H,B)$ in $\g$. For defining our quotient ultragraph $\gh$, we first modify $\g$ by an extended ultragraph $\overline{\g}$ as follows. Add the vertices $\{w':w\in B_H\setminus B\}$ to $G^0$ and denote $\overline{A}:=A\cup \{w':w\in A\cap (B_H\setminus B)\}$ for each $A\in \go$. We now define the new ultragraph $\overline{\g}=(\overline{G^0},\overline{\g}^1,r',s')$ by
\begin{align*}
\overline{G}^0&:= G^0\cup \{w':w\in B_H\setminus B\},\\
\overline{\g}^1&:=\gl,
\end{align*}
the source map
\[s'(e):=\left\{
          \begin{array}{ll}
            (s_\g(e))' & \mathrm{if}~~ s_\g(e)\in B_H\setminus B ~~\mathrm{and}~~ r_\g(e)\in H \\
            s_\g(e) & \mathrm{otherwise},
          \end{array}
        \right.
\]
and the rang map $r'(e):=\overline{r_\g(e)}$ for every $e\in \g^1$. In Proposition \ref{prop3.2} below, we will see that the $C^*$-algebras of $\g$ and $\overline{\g}$ coincide.

\begin{ex}
Suppose $\g$ is the ultragraph
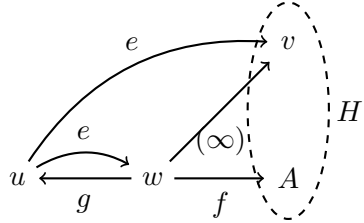
\begin{center}
\begin{tikzpicture}[scale=1.8,shorten >=1pt,auto,thick]
\node (w) at (0,0) {$w$};
\node (u) at (-1,0) {$u$};
\node (v) at (1,1) {$v$};
\node (A) at (1,0) {$A$};

\draw[->] (w) -- node[below=1pt] {$f$} (A);
\draw[->] (w) -- node[below=2pt] {$(\infty)$} (v);
\path[->] (u) edge[bend left] node[above=2pt] {$e$} (v);
\path[->] (u) edge[bend left] node[above=1pt] {$e$} (w);
\path[->] (w) edge node[below=2pt] {$g$} (u);

\draw[dashed] (1,.5) ellipse (3mm and 8mm)
node[right=5mm] {$H$};

\end{tikzpicture}
\end{center}
where $(\infty)$ indicates infinitely many edges. If $H$ is the saturated hereditary subset of $\go$ containing $\{v\}$ and $A$, then we have $B_H=\{w\}$. For $B:=\emptyset$, consider the admissible pair $(H,\emptyset)$ in $\g$. Then the ultragraph $\overline{\g}$ associated to $(H,\emptyset)$ would be
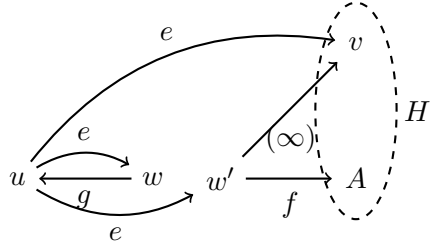
\begin{center}
\begin{tikzpicture}[scale=1.8,shorten >=1pt,auto,thick]
\node (wG) at (3.5,0) {$w$};
\node (w'G) at (4,0) {$w'$};
\node (uG) at (2.5,0) {$u$};
\node (vG) at (5,1) {$v$};
\node (AG) at (5,0) {$A$};

\draw[->] (w'G) -- node[below=2pt] {$(\infty)$} (vG);
\draw[->] (w'G) -- node[below=1pt] {$f$} (AG);
\path[->] (uG) edge[bend left] node[above=2pt] {$e$} (vG);
\path[->] (uG) edge[bend left] node[above=1pt] {$e$} (wG);
\path[->] (wG) edge node[below=.5pt] {$g$} (uG);
\path[->] (uG) edge[bend right] node[below=1pt] {$e$} (w'G);

\draw[dashed] (5,.5) ellipse (3mm and 8mm)
node[right=5mm] {$H$};
\end{tikzpicture}
\end{center}
Indeed, since $B_H\setminus B=\{w\}$, for constructing $\overline{\g}$ we first add a vertex $w'$ to $\g$. We then define
\begin{align*}
r'(f)&:=\overline{A}=A,\\
r'(e)&:=\overline{\{v,w\}}=\{v,w,w'\}, \mathrm{~ and}\\
r'(g)&:=\overline{\{u\}}=\{u\}.
\end{align*}
For the source map $s'$, for example, since $s_\g(f)\in B_H\setminus B$ and $r_\g(f)\in H$, we may define $s'(f):=w'$. Note that the range of each edge emitted by $w'$ belongs to $H$.
\end{ex}

As usual, we write $\overline{\g}^0$ for the algebra generated by the elements of $\overline{G}^0 \cup \{r'(e):e\in \overline{\g}^1\}$. Note that $\overline{A}=A$ for every $A\in H$, and hence, $H$ would be a saturated hereditary subset of $\overline{\g}^0$ as well. Moreover, the set of breaking vertices of $H$ in $\overline \g$ is $B$ (meaning $B_H^{\overline{\g}}=B$).

\begin{rem}\label{rem3.2}
Suppose that $C^*(\g)$ is generated by a Cuntz-Krieger $\g$-family $\{s_e,p_A: A\in \go,e\in\gl\}$. If a family $M=\{S_e,P_v,P_{\overline{A}}: v\in G^0,A\in \g^0, e\in \overline{\g}^1\}$ in a $C^*$-algebra $X$ satisfies relations (1)-(4) in Definition \ref{defn2.3}, we may generate a Cuntz-Krieger $\overline{\g}$-family $N=\{S_e,P_A:A\in \overline{\g}^0,e\in \overline{\g}^1\}$ in $X$. For this, since $\overline{\g}^0$ is the algebra generated by $\{v,w',r'(e): v\in G^0,w\in B_H\setminus B, e\in \overline{\g}^1\}$, it suffices to define
\begin{align*}
P_{A\cap B}&:=P_A P_B\\
P_{A\cup B}&:=P_A+P_B-P_A P_B\\
P_{A\setminus B}&:=P_A-P_A P_B
\end{align*}
and generate projections $P_A$ for all $A\in \overline{\g}^0$. Then $N$ is a Cuntz-Krieger $\overline{\g}$-family in $X$, and the $C^*$-subalgebras generated by $M$ and $N$ coincide.
\end{rem}

\begin{prop}\label{prop3.2}
Let $\g$ be an ultragraph, and let $(H,B)$ be an admissible pair in $\g$. If $\overline{\g}$ is the extended ultragraph as above, then $C^*(\g)\cong C^*(\overline{\g})$.
\end{prop}

\begin{proof}
Suppose that $C^*(\g)=C^*(t_e,q_A)$ and $C^*(\overline{\g})=C^* (s_e,p_C)$. If we define
$$\begin{array}{ll}
  P_v:=q_v & \mathrm{for}~~ v\in G^0\setminus (B_H\setminus B) \\
  P_w:=\sum_{\substack{s_\g(e)=w \\ r_\g(e)\notin H}}t_et_e^* &  \mathrm{for}~~ w\in B_H\setminus B \\
  P_{w'}:=q_w-\sum_{\substack{s_\g(e)=w \\ r_\g(e)\notin H}}t_et_e^* & \mathrm{for}~~ w\in B_H\setminus B \\
  P_{\overline{A}}:=q_A & \mathrm{for}~~ \overline{A}\in \overline{\g}^0 \\
  S_e:=t_e & \mathrm{for}~~ e\in \overline{\g}^1
\end{array}
$$
then, by Remark \ref{rem3.2}, the family $\{P_v,P_w,P_{w'},P_{\overline{A}},S_e\}$ induces a Cuntz-Krieger $\overline{\g}$-family in $C^*(\g)$. Since all vertex projections of this family are nonzero (which follows all set projections $P_A$ are nonzero for $\emptyset\neq A\in\overline{\g}^0$), the gauge-invariant uniqueness theorem \cite[Theorem 6.8]{tom} implies that the $*$-homomorphism $\phi:C^*(\overline{\g})\rightarrow C^*(\g)$ with $\phi(p_*)=P_*$ and $\phi(s_*)=S_*$ is injective. On the other hand, the family generates all $C^*(\g)$, and hence, $\phi$ is an isomorphism.
\end{proof}

To define a quotient ultragraph $\gh$, we use the following equivalent relation on $\overline{\g}$.

\begin{defn}\label{defn3.3}
Suppose that $(H,B)$ is an admissible pair in $\g$, and that $\overline{\g}$ is the extended ultragraph as above. We define the relation $\sim$ on $\overline{\g}^0$ by
$$A\sim B ~~ \Longleftrightarrow ~~ \exists V\in H ~\mathrm{such ~ that} ~ A\cup V=B\cup V.$$
Note that $A\sim B$ if and only if both sets $A\setminus B$ and $B\setminus A$ belong to $H$.
\end{defn}

\begin{lem}\label{lem3.4}
The relation $\sim$ is an equivalent relation on $\overline{\g}^0$. Furthermore, the operations
$$[A]\cup[B]:=[A\cup B], ~ [A]\cap [B]:=[A\cap B], ~ \mathrm{and}~ [A]\setminus [B]:= [A\setminus B]$$
are well-defined on the equivalent classes $\{[A]: A\in \overline{\g}^0\}$.
\end{lem}

\begin{proof}
It is straightforward to verify that the relation $\sim$ is reflexive, symmetric, and associative. To see the second assertion, suppose that $[A_1]=[A_2]$, $[B_1]=[B_2]$, and $A_1\cup V=A_2\cup V$, $B_1\cup W=B_2\cup W$ for some $V,W\in H$. The hereditary property of $H$ yields $V\cup W\in H$ and we get
\begin{align*}
(A_1 \cup B_1)\cup (V\cup W)&=(A_2 \cup B_2) \cup (V\cup W)\\
\Longrightarrow \hspace{2cm} [A_1\cup B_1]&=[A_2\cup B_2].
\end{align*}
Also, we have
\begin{align*}
(A_1 \cap B_1)\cup (V\cup W)&=(A_1 \cup (V\cup W)) \cap (B_1 \cup (V\cup W))\\
&=(A_2\cup (V\cup W))\cap (B_2 \cup (V\cup W))\\
&=(A_2\cap B_2)\cup (V\cup W)\\
\Longrightarrow \hspace{2cm} [A_1\cap B_1]&=[A_2\cap B_2].
\end{align*}

For the third operation, first note that $B_1\setminus B_2, B_2\setminus B_1 \subseteq W$ because $B_1\cup W=B_2\cup W$. Then we have
\begin{align*}
(A_1 \setminus B_1)\cup W&=\left(A_1 \setminus ((B_1\cap B_2) \cup (B_1\setminus B_2))\right)\cup W\\
&=(A_1\setminus(B_1\cap B_2))\cup W
\end{align*}
and similarly,
$$(A_2\setminus B_2)\cup W=(A_2\setminus (B_1\cap B_2))\cup W.$$
Now the fact $A_1\cup V=A_2 \cup V$ implies that
\begin{align*}
(A_1 \setminus B_1)\cup (V\cup W)&=(A_1 \setminus (B_1\cap B_2)) \cup (V\cup W)\\
&=((A_1\cup V)\setminus (B_1\cap B_2))\cup (V\cup W)\\
&=((A_2\cup V)\setminus (B_1\cap B_2))\cup (V\cup W)\\
&=(A_2\setminus B_2)\cup (V\cup W).
\end{align*}
Therefore, we obtain $[A_1 \setminus B_1]=[A_2\setminus B_2]$, as desired.
\end{proof}

\begin{defn}
Let $\g$ be an ultragraph, let $(H,B)$ be an admissible pair in $\g$, and consider the equivalent relation of Definition \ref{defn3.3} on the extended ultragraph $\overline{\g}=(\overline{G}^0,\overline{\g}^1,r',s')$. The \emph{quotient ultragraph of $\g$ by $(H,B)$} is the quintuple $\gh=(\pGo, \pgo, \pgl,r,s)$, where
\begin{align*}
\Phi(G^0)&:=\left\{[v]:v\in G^0\setminus H\right\}\cup \left\{[w']:w\in B_H\setminus B\right\},\\
\Phi(\go)&:=\left\{[A]:A\in \overline{\g}^0 \right\},\\
\Phi(\gl)&:=\left\{e\in \overline{\g}^1:r'(e)\notin H \right\},
\end{align*}
and $r:\pgl \rightarrow \pgo$, $s:\pgl \rightarrow \pGo$ are the range and source maps defined by
$$r(e)=[r'(e)] \hspace{5mm} \mathrm{and} \hspace{5mm} s(e):=[s'(e)].$$
We refer to $\pGo$ as the vertices of $\gh$.
\end{defn}

\begin{rem}
Lemma \ref{lem3.4} implies that $\pgo$ is the smallest algebra containing
$$\left\{[v],[w']:v\in G^0\setminus H, w\in B_H\setminus B\right\}\cup \left\{[r'(e)]: e\in \overline{\g}^1\right\}.$$
\end{rem}

{\bf Notation.}\begin{enumerate}
\item For every vertex $v\in \overline{\g}^0\setminus H$, we usually denote $[v]$ instead of $[\{v\}]$.
\item For $A,B\in\overline{\g}^0$, we write $[A]\subseteq [B]$ whenever $[A]\cap [B]=[A]$.
\item Through the paper, we will denote the range and the source maps of $\g$ by $r_\g ,s_\g$, those of $\overline{\g}$ by $r',s'$, and those of $\gh$ by $r,s$.
\end{enumerate}

Now we introduce representations of quotient ultragraphs and their relative $C^*$-algebras.

\begin{defn}\label{defn3.8}
Let $\gh$ be a quotient ultragraph. A \emph{representation of $\gh$} is a set of partial isometries $\{T_e:e\in \pgl\}$ and a set of projections $\{Q_{[A]}: [A]\in \pgo\}$ which satisfy the following relations:
\begin{itemize}
\item[(1)]{$Q_{[\emptyset]}=0$, $Q_{[A\cap B]}=Q_{[A]} Q_{[B]}$, and $Q_{[A\cup B]}=Q_{[A]} +Q_{[B]}-Q_{[A\cap B]}$.}
\item[(2)]{$T_e^* T_e=Q_{r(e)}$ and $T_e^* T_f=0$ when $e\neq f$.}
\item[(3)]{$T_e T_e^*\leq Q_{s(e)}$.}
\item[(4)]{$Q_{[v]}=\sum_{s(e)=[v]}T_e T_e^*$, whenever $0<|s^{-1}([v])|<\infty$.}
\end{itemize}
The $C^*$-algebra $C^*(\g/(H,B))$ of $\g/(H,B)$ is the universal $C^*$-algebra generated by a representation $\{t_e,q_{[A]}:[A]\in \pgo,e\in \pgl\}$ which exists by Theorem \ref{thm3.11} below.
\end{defn}

\begin{rem}
Recall that the universality of $C^*(\g/(H,B))$ means that if $\{T_e,Q_{[A]}\}$ is a representation of $\g/(H,B)$ in a $C^*$-algebra $X$, then there exists a $*$-homomorphism $\phi:C^*(\g/(H,B))\rightarrow X$ such that $\phi(q_{[A]})=Q_{[A]}$ and $\phi(t_e)=T_e$ for all $[A]\in \pgo$ and $e\in \pgl$.
\end{rem}

Note that if $\alpha=e_1\ldots e_n$ is a path in $\overline{\g}$ and $r'(\alpha)\notin H$, then the hereditary property of $H$ yields $r'(e_i)\notin H$, and so $e_i\in \pgl$ for all $1\leq i\leq n$. In this case, we denote $t_\alpha:=t_{e_1}\ldots t_{e_n}$. Moreover, we define
$$(\gh)^*:=\left\{[A]: [A]\neq [\emptyset]\right\}\cup \left\{\alpha\in\overline{\g}^*: r(\alpha)\neq [\emptyset]\right\}$$
as the set of finite paths in $\gh$ and we can extend the maps $s,r$ on $(\gh)^*$ by setting
$$s([A]):=r([A]):=[A] ~~~ \mathrm{and} ~~~ s(\alpha):=s(e_1),~ r(\alpha):=r(e_n). $$

The proof of next lemma is similar to the arguments of \cite[Lemmas 2.8 and 2.9]{tom}.

\begin{lem}\label{lem3.10}
Let $\gh$ be a quotient ultragraph and let $\{T_e,Q_{[A]}\}$ be a representation of $\gh$. Then any nonzero word in $T_e$, $Q_{[A]}$, and $T_f^*$ may be written as a finite linear combination of the forms $T_\alpha Q_{[A]} T_\beta^*$ for $\alpha,\beta\in (\gh)^*$ and $[A]\in \pgo$ with $[A]\cap r(\alpha)\cap r(\beta)\neq [\emptyset]$.
\end{lem}

\begin{thm}\label{thm3.11}
Let $\gh$ be a quotient ultragraph. Then there exists a (unique up to isomorphism) $C^*$-algebra $C^*(\gh)$ generated by a universal representation $\{t_e,q_{[A]}: [A]\in \pgo,e\in \pgl\}$ for $\gh$. Furthermore, all the $t_e$'s and $q_{[A]}$'s are nonzero for $[\emptyset]\neq [A]\in \pgo$ and $ e\in \pgl$.
\end{thm}

\begin{proof}
By a standard argument similar to the proof of \cite[Theorem 2.11]{tom}, we may construct such universal $C^*$-algebra $C^*(\gh)$. Note that the universality implies that $C^*(\gh)$ is unique up to isomorphism. To show the last statement, we generate a representation for $\gh$ as follows. Suppose $C^*(\overline{\g})=C^*(s_e,p_A)$ and consider $I_{(H,B)}$ as an ideal of $C^*(\overline{\g})$ by the isomorphism of Proposition \ref{prop3.2}. If we define
$$\left\{
    \begin{array}{ll}
      Q_{[A]}:=p_A+I_{(H,B)} & \mathrm{for~~}[A]\in \pgo \\
      T_{e}:=s_e+I_{(H,B)} & \mathrm{for}~~ e\in \pgl,
    \end{array}
  \right.
$$
then the family $\{T_{e},Q_{[A]}:[A]\in \pgo,e\in \pgl\}$ is a representation for $\gh$ in the quotient $C^*$-algebra $C^*(\overline{\g})/I_{(H,B)}$. Note that the definition of $Q_{[A]}$'s is well-defined. Indeed, if $A_1\cup V=A_2 \cup V$ for some $V\in H$, then $p_{A_1}+p_{V\setminus A_1}=p_{A_2}+p_{V\setminus A_2}$ and hence $p_{A_1}+I_{(H,B)}=p_{A_2}+I_{(H,B)}$ by the facts $V\setminus A_1, V\setminus A_2\in H$.

Moreover, all elements $Q_{[A]}$ and $T_e$ are nonzero for $[\emptyset]\ne[A]\in\pgo$, $e\in\pgl$. In fact, if $Q_{[A]}=0$, then $p_A\in I_{(H,B)}$ and we get $A\in H$ by \cite[Theorem 6.12]{kat3}. Also, since $T_e^* T_e=Q_{r(e)}\ne0$, all partial isometries $T_e$ are nonzero.

Now suppose that $C^*(\gh)$ is generated by the family $\{t_e,q_{[A]}:[A]\in \pgo, ~ e\in \pgl\}$. By the universality of $C^*(\gh)$, there is a $*$-homomorphism $\phi:C^*(\gh)\rightarrow C^*(\overline{\g})/I_{(H,B)}$ such that $\phi(t_e)=T_e$ and $\phi(q_{[A]})=Q_{[A]}$, and thus, all elements $\{t_e,q_{[A]}:[\emptyset]\ne[A]\in \pgo, ~ e\in \pgl\}$ are nonzero.
\end{proof}

Using Lemma \ref{lem3.10}, one may easily show that
$$C^*(\gh)=\overline{\mathrm{span}}\big\{t_\alpha q_{[A]}t_\beta^*: \alpha,\beta\in (\gh)^*,~ r(\alpha)\cap [A]\cap r(\beta)\neq [\emptyset]\big\}.$$


\section{Uniqueness Theorems}

After defining the $C^*$-algebras of quotient ultragraphs, in this section, we prove the gauge invariant and the Cuntz-Krieger uniqueness theorems for them. To do this, we approach to a quotient ultragraph $C^*$-algebra by graph $C^*$-algebras and then apply the corresponding uniqueness theorems of graph $C^*$-algebras. This approach is a developed version of the dual graph method of \cite[Section 2]{rae} and \cite[Section 5]{tom} with more complications.

We fix again an ultragraph $\g$, an admissible pair $(H,B)$ in $\g$, and the quotient ultragraph $\gh=(\pGo,\pgo,\pgl,r,s)$.

\begin{defn}
We say that a vertex $[v]\in \pGo$ is a \emph{sink} if $s^{-1}([v])=\emptyset$. If $[v]$ emits finitely many edges of $\pgl$, $[v]$ is called a \emph{regular vertex}. We say $[v]$ is a \emph{singular vertex} if it is not regular. We denote the set of all singular vertices of $\pGo$ by
\[\sg:=\big\{[v]\in \pGo: |s^{-1}([v])|=0~ \mathrm{or}~\infty\big\}.\]
\end{defn}

Let $F$ be a finite subset of $\sg \cup \pgl$, and denote $F^0:=F\cap \sg$ and $F^1:=F\cap \pgl=\{e_1,\ldots,e_n\}$. We want to construct a special graph $G_F$ such that $C^*(G_F)$ is isomorphic to $C^*(t_e,q_{[v]}:[v]\in F^0,e\in F^1)$. For each $\omega=(\omega_1,\ldots, \omega_n)\in \{0,1\}^n\setminus \{0^n\}$, we write
$$r(\omega):=\bigcap_{\omega_i=1}r(e_i)\setminus \bigcup_{\omega_j=0}r(e_j) \mathrm{~~and~~} R(\omega):=r(\omega)\setminus \bigcup F^0.$$
Note that $r(\omega)\cap r(\nu)=[\emptyset]$ for distinct $\omega,\nu\in \{0,1\}\setminus \{0^n\}$. If
\begin{multline*}
    \Gamma_0:=\big\{\omega\in \{0,1\}^n\setminus\{0^n\}: \exists [v_1],\ldots,[v_m]\in\pgo $ such  that $\\
 R(\omega)=\bigcup_{i=1}^m[v_i] $ and $ \emptyset\neq s^{-1}([v_i])\subseteq F^1 $ for $ 1\leq i\leq m\big\},
\end{multline*}
we consider the finite set
\[\Gamma:=\left\{\omega\in \{0,1\}^n\setminus \{0^n\}: R(\omega)\neq [\emptyset] \mathrm{~and ~} \omega\notin \Gamma_0 \right\}.\]

Now we define the finite graph $G_F=(G_F^0,G_F^1,r_F,s_F)$ containing the vertices $G_F^0:=F^0 \cup F^1 \cup \Gamma$ and the edges
\begin{align*}
G_F^1:=&\left\{(e,f)\in F^1\times F^1: s(f)\subseteq r(e) \right\}\\
       &\cup \left\{(e,[v])\in F^1\times F^0: [v]\subseteq r(e) \right\}\\
       &\cup \left\{(e,\omega)\in F^1\times \Gamma: \omega_i=1 \mathrm{~~when~~} e=e_i \right\}
\end{align*}
with the source map $s_F(e,f)=s_F(e,[v])=s_F(e,\omega)=e$, and the range map $r_F(e,f)=f$, $r_F(e,[v])=[v]$, $r_F(e,\omega)=\omega$.

\begin{prop}\label{prop4.2}
Let $\gh$ be a quotient ultragraph and let $F$ be a finite subset of $\sg \cup \pgl$. If $C^*(\gh)=C^*(t_e,q_{[A]})$, then the elements
$$\begin{array}{ccc}
    Q_e:=t_et_e^*, & Q_{[v]}:=q_{[v]}(1-\sum_{e\in F^1}t_et_e^*), & Q_\omega:=q_{R(\omega)}(1-\sum_{e\in F^1}t_et_e^*) \\
    T_{(e,f)}:=t_eQ_f, & T_{(e,[v])}:=t_e Q_{[v]}, & T_{(e,\omega)}:=t_e Q_\omega
  \end{array}
$$
form a Cuntz-Krieger $G_F$-family generating the $C^*$-subalgebra $C^*(t_e,q_{[v]}:[v]\in F^0,e\in F^1)$ of $C^*(\gh)$. Moreover, all projections $Q_*$ are nonzero.
\end{prop}

\begin{proof}
We first note that all the projections $Q_e$, $Q_{[v]}$, and $Q_\omega$ are nonzero. Indeed, each $[v]\in F^0$ is a singular vertex in $\gh$, so $Q_{[v]}$ is nonzero. Also, by definition, for every $\omega\in \Gamma$ we have $\omega\notin \Gamma_0$ and $R(\omega)\neq [\emptyset]$. Hence, for any $\omega\in \Gamma$, if there is an edge $f\in \pgl\setminus F^1$ with $s(f)\subseteq R(\omega)$, then $0\neq t_ft_f^*\leq Q_{\omega}$. If there is a sink $[w]$ such that $[w]\subseteq R(\omega)=r(\omega)\setminus \bigcup F^0$, then $0\neq q_{[w]}\leq q_{R(\omega)}(1-\sum_{e\in F^1}t_et_e^*)=Q_{\omega}$. Thus $Q_{\omega}$ is nonzero in either case. In addition, the projections $Q_e$, $Q_{[v]}$, and $Q_\omega$ are mutually orthogonal because of the factor $1-\sum_{e\in F^1}t_et_e^*$ and the definition of $R(\omega)$.

Now we show the collection $\{T_x,Q_a:a\in G_F^0, x\in G_F^1\}$ is a Cuntz-Krieger $G_F$-family by checking the relations (1)-(3) in Remark \ref{rem2.4}.

\underline{(1)}: Since $Q_{[v]},Q_\omega \leq q_{r(e)}$ for $(e,[v]),(e,\omega)\in G_F^1$, we have
$$T^*_{(e,f)}T_{(e,f)}=Q_f t_e^* t_e Q_f=t_f t_f^* q_{r(e)} t_f t_f^*= t_f q_{r(f)} t_f^*=Q_f,$$
$$T^*_{(e,[v])}T_{(e,[v])}=Q_{[v]}t_e^* t_e Q_{[v]}=Q_{[v]} q_{r(e)} Q_{[v]}=Q_{[v]},$$
and
$$T^*_{(e,\omega)}T_{(e,\omega)}=Q_\omega t_e^* t_e Q_\omega=Q_\omega q_{r(e)} Q_\omega= Q_\omega.$$

\underline{(2)}: This relation may be checked similarly.

\underline{(3)}: Note that any element of $F^0 \cup \Gamma$ is a sink in $G_F$. So, fix some $e_i\in F^1$ as a vertex of $G_F^0$. Write $q_{F^0}:=\sum_{[v]\in F^0}q_{[v]}$. We compute
\begin{enumerate}[(i)]
\item
$$q_{r(e_i)} \sum_{\substack{f\in F^1 \\ s(f)\subseteq r(e_i)}}Q_f=q_{r(e_i)}\sum_{\substack{f\in F^1 \\ s(f)\subseteq r(e_i)}} t_f t_f^*=q_{r(e_i)}\sum_{f\in F^1}t_f t_f^*;$$
\item
\begin{align*}
q_{r(e_i)}\sum_{\substack {[v]\in F^0,\\ [v]\subseteq r(e_i)}}Q_{[v]}&=q_{r(e_i)}\sum_{[v]\in F^0}q_{[v]}(1-\sum_{e\in F^1}t_e t_e^*)\\
&=q_{r(e_i)} q_{F^0}(1-\sum_{e\in F^1}t_e t_e^*);
\end{align*}
\item
$$\sum_{\omega\in \Gamma,\omega_i=1}Q_\omega =\sum_{\omega\in \Gamma,\omega_i=1}q_{R(\omega)}(1-\sum_{e\in F^1}t_e t_e^*)=\sum_{\omega_i=1}q_{R(\omega)}(1-\sum_{e\in F^1}t_e t_e^*),$$
because $\sum_{\omega_i=1}q_{R(\omega)}=q_{r(e_i)}(1-q_{F^0})$.
\end{enumerate}
We can use these relations to get
\begin{align*}
\sum_{s(f)\subseteq r(e_i)}&T_{(e_i,f)}+\sum_{[v]\in F^0,~[v]\subseteq r(e_i)}T_{(e_i,[v])}+ \sum_{\omega\in \Gamma,~\omega_i=1}T_{(e_i,\omega)}\\
&=t_{e_i} \left(q_{r(e_i)}\sum_{e\in F^1}t_et_e^*+q_{r(e_i)} q_{F^0}(\sum_{e\in F^1}t_et_e^*)+ q_{r(e_i)}(1-q_{F^0})(\sum_{e\in F^1}t_et_e^*) \right)\\
&=t_{e_i}q_{r(e_i)}\left(\sum_{e\in F^1}t_et_e^*+(q_{F^0}+1-q_{F^0})(1-\sum_{e\in F^1}t_et_e^*) \right)\\
&=t_{e_i}.
\end{align*}
\begin{flushright}
        (4.1)
\end{flushright}
Now if $e_i$ is not a sink as a vertex in $G_F$ (i.e. $|\{x\in G_F^1:s_F(x)=e_i\}|>0$), we conclude that
\begin{align*}
\sum_{f\in F^1,~s(f)\subseteq r(e_i)}&T_{(e_i,f)}T_{(e_i,f)}^*+\sum_{[v]\in F^0,~[v]\subseteq r(e_i)}T_{(e_i,[v])}T_{(e_i,[v])}^*+ \sum_{\omega\in \Gamma,~\omega_i=1}T_{(e_i,\omega)}T_{(e_i,\omega)}^*\\
&=\sum t_{e_i}Q_f t_{e_i}^*+\sum t_{e_i}Q_{[v]} t_{e_i}^*+\sum t_{e_i}Q_\omega t_{e_i}^*\\
&=t_{e_i}q_{r(e_i)}(\sum Q_f +\sum Q_{[v]}+\sum Q_\omega)t_{e_i}^*\\
&=t_{e_i}t_{e_i}^*=Q_{e_i}.
\end{align*}
which establishes the relation (3).

Furthermore, equation (4.1) says that $t_{e_i}\in C^*(T_*,Q_*)$ for every $e_i\in F^1$. Also, for each $[v]\in F^0$, we have
\begin{align*}
Q_{[v]}+\sum_{e\in F^1, s(e)=[v]} Q_e&=t_{[v]}(1-\sum_{e\in F^1}t_et_e^*)+\sum_{e\in F^1, s(e)=[v]}t_e t_e^*\\
&=t_{[v]}-t_{[v]}\sum_{e\in F^1} t_et_e^*+t_{[v]}\sum_{e\in F^1} t_et_e^*\\
&=t_{[v]}.
\end{align*}
Therefore, the family $\{T_x,Q_a:a\in G_F^0, x\in G_F^1\}$ generates the $C^*$-subalgebra $C^*(\{t_e,q_{[v]}:e\in F^1, [v]\in F^0\})$ of $C^*(\gh)$ and the proof is complete.
\end{proof}

\begin{cor}\label{cor4.3}
If $F$ is a finite subset of $\sg \cup \pgl$, then $C^*(G_F)$ is isometrically isomorphic to the $C^*$-subalgebra of $C^*(\gh)$ generated by $\{t_e,q_{[v]}: [v]\in F^0, e\in F^1\}$.
\end{cor}

\begin{proof}
Suppose that $X$ is the $C^*$-subalgebra generated by $\{t_e,q_{[v]}: [v]\in F^0, e\in F^1\}$ and let $\{T_x,Q_a: a\in G_F^0,x\in G_F^1\}$ be the Cuntz-Krieger $G_F$-family in Proposition \ref{prop4.2}. If $C^*(G_F)=C^*(s_x,p_a)$, then there exists a $*$-homomorphism $\phi:C^*(G_F)\rightarrow X$ with $\phi(p_a)=Q_a$ and $\phi(s_x)=T_x$ for every $a\in G_F^0$, $x\in G_F^1$. Since each $Q_a$ is nonzero by Proposition \ref{prop4.2}, the gauge invariant uniqueness theorem implies that $\phi$ is injective. Moreover, the family $\{T_x,Q_a\}$ generates $X$, so $\phi$ is an isomorphism.
\end{proof}

Note that if $F_1\subseteq F_2$ are two finite subsets of $\sg \cup \pgl$ and $X_1,X_2$ are the $C^*$-subalgebras of $C^*(\gh)$ corresponding to $G_{F_1}$ and $G_{F_2}$, respectively, we then have $X_1\subseteq X_2$ by Proposition \ref{prop4.2}.

\begin{rem}
Using the relations in Definition \ref{defn3.8}, each $q_{[A]}$ for $[A]\in \pGo$, can be produced by the elements of
$$\{q_{[v]}:[v]\in \sg\} \cup \{t_e:e\in \pgl\}$$
with finitely many operations. So, the $*$-subalgebra of $C^*(\gh)$ generated by
$$\{q_{[v]}:[v]\in \sg\} \cup \{t_e:e\in \pgl\}$$
is dense in $C^*(\gh)$.
\end{rem}

As for graph $C^*$-algebras, we can apply the universal property to have a strongly continuous \emph{gauge action} $\gamma:\mathbb{T}\rightarrow \mathrm{Aut}(C^*(\gh))$ such that
$$\gamma_z(t_e)=zt_e ~~~ \mathrm{and} ~~~ \gamma_z(q_{[A]})=q_{[A]}$$
for every $[A]\in \pgo$, $e\in \pgl$, and $z\in\mathbb{T}$. Now we are ready to prove the uniqueness theorems.

\begin{thm}[The Gauge Invariant Uniqueness Theorem]\label{thm4.5}
Let $\gh$ be a quotient ultragraph and let $\{T_e,Q_{[A]}\}$ be a representation for $\gh$ such that $Q_{[A]}\neq 0$ for $[A]\neq [\emptyset]$. If $\pi_{T,Q}:C^*(\gh)\rightarrow C^*(T_e,Q_{[A]})$ is the $*$-homomorphism satisfying $\pi_{T,Q}(t_e)=T_e$, $\pi_{T,Q}(q_{[A]})=Q_{[A]}$, and there is a strongly continuous action $\beta$ of $\mathbb{T}$ on $C^*(T_e,Q_{[A]})$ such that $\beta_z\circ \pi_{T,Q}=\pi_{T,Q}\circ \gamma_z$ for every $z\in \mathbb{T}$, then $\pi_{T,Q}$ is faithful.
\end{thm}

\begin{proof}
Select an increasing sequence $\{F_n\}$ of finite subsets of $\sg \cup \pgl$ such that $\cup_{n=1}^\infty F_n=\sg\cup \pgl$. For each $n$, Corollary \ref{cor4.3} gives an isomorphism
$$\pi_n:C^*(G_{F_n})\rightarrow C^*(\{t_e,q_{[v]}:[v]\in F^0,e\in F^1\})$$
that respects the generators. We can apply the gauge invariant uniqueness theorem for graph $C^*$-algebras to see that the homomorphism
$$\pi_{T,Q}\circ \pi_n:C^*(G_{F_n})\rightarrow C^*(T_e,Q_{[A]})$$
is faithful. Hence, for every $F_n$, the restriction of $\pi_{T,Q}$ on the $*$-subalgebra of $C^*(\gh)$ generated by $\{t_e,q_{[v]}:[v]\in F_n^0, e\in F_n^1\}$ is faithful. This turns out that $\pi_{T,Q}$ is injective on the $*$-subalgebra $C^*(t_e,q_{[v]}:[v]\in \sg, e\in \pgl)$. Since, this subalgebra is dense in $C^*(\gh)$, we conclude that $\pi_{T,Q}$ is faithful.
\end{proof}

To prove a verson of Cuntz-Krieger uniqueness theorem, we extend Condition (L) for quotient ultragraphs.

\begin{defn}
We say that $\gh$ satisfies \emph{Condition (L)} if for every loop $\alpha=e_1\ldots e_n$ in $\gh$, at least one of the following conditions holds:
\begin{enumerate}[(i)]
  \item $r(e_i)\neq s(e_{i+1})$ for some $1\leq i\leq n$, where $e_{i+1}:=e_1$ (or equivalently, $r(e_i)\setminus s(e_{i+1})\ne [\emptyset]$).
  \item $\alpha$ has an exit; that means, there exists $f\in \pgl$ such that $s(f)\subseteq r(e_i)$ and $f\neq e_{i+1}$ for some $1\leq i\leq n$.
\end{enumerate}
\end{defn}

\begin{lem}\label{lem4.7}
Let $F$ be a finite subset of $\sg \cup \pgl$. If $\gh$ satisfies Condition (L), so does the graph $G_F$.
\end{lem}

\begin{proof}
Suppose that $\gh$ satisfies Condition (L). As the elements of $F^0\cup \Gamma$ are sinks in $G_F$, every loop in $G_F$ is of the form $\widetilde{\alpha}=(e_1,e_2)\ldots (e_n, e_1)$ corresponding with a loop $\alpha=e_1\ldots e_n$ in $\gh$. So, fix a loop $\widetilde{\alpha}=(e_1,e_2)\ldots (e_n, e_1)$ in $G_F$. Then $\alpha=e_1\ldots e_n$ is a loop in $\gh$ and by Condition (L), one of the following holds:
\begin{enumerate}[(i)]
\item $r(e_i)\neq s(e_{i+1})$ for some $1\leq i\leq n$, where $e_{i+1}:=e_1$, or
\item there exists $f\in \pgl$ such that $s(f)\subseteq r(e_i)$ and $f\neq e_{i+1}$ for some $1\leq i\leq n$.
\end{enumerate}

We can suppose in the case (i) that $s(e_{i+1})\subsetneq r(e_i)$ and $r(e_i)$ emits only the edge $e_{i+1}$ in $\gh$. Then, by the definition of $\Gamma$, there exists either $[v]\in F^0$ with $[v]\subseteq r(e_i)\setminus s(e_{i+1})$, or $\omega \in \Gamma $ with $\omega_i=1$. Thus, either $(e_i,[v])$ or $(e_i,\omega)$ is an exit for the loop $\widetilde{\alpha}$ in $G_F$, respectively.

Now assume case (ii) holds. If $f\in F^1$, then $(e_i,f)$ is an exit for $\widetilde{\alpha}$. If $f\notin F^1$, for $[v]:=s(f)$ we have either $[v]\notin F^0$ or
$$\exists \omega\in \Gamma ~ \mathrm{with} ~ \omega_i=1 ~ \mathrm{such~that} ~ [v]\subseteq R(\omega).$$
Hence, $(e_i,[v])$ or $(e_i,\omega)$ is an exit for $\widetilde{\alpha}$, respectively. Consequently, in any case, $\widetilde{\alpha}$ has an exit.
\end{proof}

\begin{thm}[The Cuntz-Krieger Uniqueness Theorem]\label{thm4.8}
Suppose that $\gh$ is a quotient ultragraph satisfying Condition (L). If $\{T_e,Q_A\}$ is a Cuntz-Krieger representation for $\gh$ in which all the projection $Q_{[A]}$ are nonzero for $[A]\neq [\emptyset]$, then the $*$-homomorphism $\pi_{T,Q}:C^*(\gh)\rightarrow C^*(T_e,Q_{[A]})$ with $\pi_{T,Q}(t_e)=T_e$ and $\pi_{T,Q}(q_{[A]})=Q_{[A]}$ is an isometrically isomorphism.
\end{thm}

\begin{proof}
It suffices to show that $\pi_{T,Q}$ is faithful. Similar to the proof of \ref{thm4.5}, choose an increasing sequence $\{F_n\}$ of finite sets such that $\cup_{n=1}^\infty F_n=\sg \cup \pgl$. By Corollary \ref{cor4.3}, there are isomorphisms $\pi_n: C^*(G_{F_n})\rightarrow C^*\left(\{t_e, q_{[v]}:[v]\in F_n^0,e\in F_n^1\}\right)$ that respect the generators. Since all the graphs $G_{F_n}$ satisfy Condition (L) by Lemma \ref{lem4.7}, the Cuntz-Krieger uniqueness theorem for graph $C^*$-algebras implies that the $*$-homomorphisms
$$\pi_{T,Q}\circ \pi_n: C^*(G_{F_n})\rightarrow C^*(T_e,Q_{[A]})$$
are faithful. Therefore, $\pi_{T,Q}$ is faithful on the subalgebra $C^*(t_e,q_{[v]}:[v]\in \sg, e\in \pgl)$ of $C^*(\gh)$. Since this subalgebra is dense in $C^*(\gh)$, we conclude that $\pi_{T,Q}$ is a faithful homomorphism.
\end{proof}


\section{Gauge Invariant Ideals}

Here we want to describe the gauge invariant ideals of a quotient ultragraph $C^*$-algebra $C^*(\gh)$ to apply in Sections 6 and 7. For this, we first prove that $C^*(\gh)$ is isomorphic to the quotient $C^*$-algebra $C^*(\g)/I_{(H,B)}$, and then use the ideal structure of $C^*(\g)$ \cite[Section 6]{kat3}.

\begin{prop}\label{prop5.1}
Let $\g$ be an ultragraph. If $(H,B)$ is an admissible pair in $\g$, then $C^*(\gh)\cong C^*(\g)/I_{(H,B)}$.
\end{prop}

\begin{proof}
Using Proposition \ref{prop3.2}, we can consider $I_{(H,B)}$ as an ideal of $C^*(\overline{\g})$. Suppose that $C^*(\overline{\g})=C^*(s_e,p_A)$ and $C^*(\gh)=C^*(t_e,q_{[A]})$. If we define
$$T_e:=s_e+I_{(H,B)} ~~~ \mathrm{and} ~~~ Q_{[A]}:=p_A+I_{(H,B)}$$
for every $[A]\in \pgo$ and $e\in \pgl$, then the family $\{T_e,Q_{[A]}\}$ is a representation for $\gh$ in $C^*(\overline{\g})/I_{(H,B)}$. So, there is a $*$-homomorphism $\phi:C^*(\gh)\rightarrow C^*(\g)/I_{(H,B)}$ such that $\phi(t_e)=T_e$ and $\phi(q_{[A]})=Q_{[A]}$. Moreover, all $Q_{[A]}$ with $[A]\neq [\emptyset]$ are nonzero because $p_A+I_{(H,B)}=I_{(H,B)}$ implies $A\in H$. Then, an application of Theorem \ref{thm4.5} yields that $\phi$ is faithful. On the other hand, the family $\{T_e,Q_{[A]}:[A]\in\pgo,e\in\pgl\}$ generates the quotient $C^*(\g)/I_{(H,B)}$, and hence, $\phi$ is surjective as well. Therefore, $\phi$ is an isomorphism and the result follows.
\end{proof}

In the rest of section, we fix a quotient ultragraph $\gh$ and assume that $C^*(\g)=C^*(s_e,p_A)$ and $C^*(\gh)=C^*(t_e,q_{[A]})$. If we order the admissible pairs in $\g$ by
$$(K_1,S_1)\preceq (K_2,S_2) ~~ \Longleftrightarrow ~~  K_1\subseteq K_2 ~~ \mathrm{and} ~~ S_1\subseteq K_2 \cup S_2,$$
then \cite[Theorem 6.12]{kat3} shows that the map $(K,S) \mapsto I_{(K,S)}$ is a one-to-one order preserving correspondence between admissible pairs in $\g$ and gauge invariant ideals of $C^*(\g)$. Moreover, for every gauge invariant ideal $I_{(K,S)}$ in $C^*(\g)$, we have
$\{A\in \go: p_A\in I_{(K,S)}\}=K$ and $\{w\in B_K: p_w^K\in I_{(K,S)}\}=S$.

\begin{lem}\label{lem5.2}
Let $\phi: C^*(\g)\rightarrow C^*(\g/(H,B))$ be the canonical surjection described in Proposition \ref{prop5.1}, and let $(K,S)$ be an admissible pair in $\g$. Then the image of ideal $I_{(K,S)}$ of $C^*(\g)$ under $\phi$ is the ideal $J_{[K,S]}$ in $C^*(\g/(H,B))$ generated by
\begin{multline*}
\left\{q_{[\overline{A}]},q_{[w']}: A\in K, w\in S\cap(B_H\setminus B)\right\}\\
\cup\left\{q_{[w]}-\sum_{s_\g(e)=w, r_\g(e)\notin K} t_et_e^* :w\in S\setminus (B_H\setminus B)\right\}.
\end{multline*}
\end{lem}

\begin{proof}
Using the isomorphisms in Propositions \ref{prop3.2} and \ref{prop5.1}, we have $\phi(p_A)=q_{[\overline{A}]}$ and $\phi(p_w^H)=q_{[w']}$ for every $A\in K$ and $w\in S\cap (B_H\setminus B)$. Also, if $w\in S\setminus (B_H\setminus B)$, then
$$\phi(p_w-\sum_{s_\g(e)=w, r_\g(e)\notin K} s_es_e^*)=q_{[w]}-\sum_{s_\g(e)=w, r_\g(e)\notin K} t_et_e^*.$$
These follow the result.
\end{proof}

\begin{thm}\label{thm5.3}
Let $\gh$ be a quotient ultragraph of $\g$. Then
\begin{enumerate}
  \item the map $(K,S) \mapsto J_{[K,S]}$ is a one-to-one order preserving correspondence between the admissible pairs $(K,S)$ in $\g$ with $H\subseteq K$, $B\subseteq K\cup S$ and the gauge invariant ideals of $C^*(\gh)$;
  \item if $J_{[K,S]}$ is a gauge invariant ideal in $C^*(\gh)$, then
$$\frac{C^*(\gh)}{J_{[K,S]}}\cong C^*(\g/(K,S)).$$
\end{enumerate}
\end{thm}

\begin{proof}
(1): Suppose that $J$ is a gauge invariant ideal in $C^*(\gh)$. If $\phi: C^*(\g)\rightarrow C^*(\gh)$ is the canonical map of Lemma \ref{lem5.2} and we define $I:=\phi^{-1}(J)$, then $I$ is a gauge invariant ideal of $C^*(\g)$ with $I_{(H,B)}\subseteq I$. By \cite[Theorem 6.12]{kat3}, there exists an admissible pair $(K,S)$ in $\g$ such that $I=I_{(K,S)}$. In particular, we have $H\subseteq K$ and $B\subseteq K\cup S$. Now Lemma \ref{lem5.2} implies that $J=\phi(I_{(K,S)})=J_{[K,S]}$, as demanded.

(2): Let $J_{[K,S]}$ be a gauge invariant ideal in $C^*(\gh)$. Since $\phi(I_{(K,S)})=J_{[K,S]}$ by Lemma \ref{lem5.2}, Proposition \ref{prop5.1} yields that
$$\frac{C^*(\gh)}{J_{[K,S]}}\cong \frac{C^*(\g)/I_{(H,B)}}{I_{(K,S)}/I_{(H,B)}}\cong \frac{C^*(\g)}{I_{(K,S)}}\cong C^*(\g/(K,S)).$$
\end{proof}

\section{Condition (K)}

It is known that a graph $E$ satisfies Condition (K) if and only if every ideal in $C^*(E)$ is gauge invariant (see \cite{bat,dri} among others). In \cite[Section 7]{kat3}, the authors proved a similar result for ultragraph $C^*$-algebras by applying the $C^*$-algebras of topological graphs. In this section, we extend Condition (K) for a quotient ultragraph $\gh$ under which all ideals of $C^*(\gh)$ are gauge invariant. In particular, we can alternatively obtain \cite[Proposition 7.3]{kat3} by applying the quotient ultragraphs. To do this, we first see that a quotient ultragraph $\gh$ satisfies Condition (K) if and only if every quotient ultragraph $\g/(K,S)$ with $H\subseteq K$ satisfies Condition (L). Then the main results will be shown.

Let $\alpha=e_1\ldots e_n$ be a path in an ultragraph $\g$. For $1\leq k\leq l\leq n$, the path $\beta=e_k e_{k+1}\ldots e_l$ is called a \emph{subpath of $\alpha$}. We simply write $\beta\subseteq \alpha$ when $\beta$ is a subpath of $\alpha$; otherwise, we write $\beta\nsubseteq \alpha$.

\begin{defn}
We say that a quotient ultragraph $\gh$ of $\g$ satisfies \emph{Condition (K)} if every vertex $v\in G^0\setminus H$ either is the base of no loops, or there are at least two loops $\alpha,\beta$ in $\g$ based at $v$ such that neither $\alpha$ nor $\beta$ is a subpath of the other.
\end{defn}

\begin{rem}
If $\g/(\emptyset,\emptyset)$ is an ultragraph, the above definition for Condition (K) coincide with that of \cite[Definition 7.1]{kat3}.
\end{rem}

In the following, we show in the absence of Condition (K) for $\gh$ that there is a quotient ultragraph $\g/(K,S)$ with $K\supseteq H$ such that it does not satisfy Condition (L). To d this, let $\gh$ be a quotient ultragraph not satisfying Condition (K). Then $\gh$ contains a loop $\gamma=e_1\ldots e_n$ such that there are no loops $\alpha$ with $s(\alpha)=s(\gamma)$, $\alpha\nsubseteq \gamma$, and $\gamma\nsubseteq\alpha$. If $\gamma^0:=\{s_\g(e_1),\ldots, s_\g(e_n)\}$, define
$$X:=\left\{r_\g(\alpha)\setminus \gamma^0: \alpha\in \g^*,|\alpha|\geq 1, s_\g(\alpha)\in \gamma^0\right\},$$
$$Y:=\left\{\bigcup_{i=1}^n A_i: A_1,\ldots,A_n\in X \cup H, n\in \mathbb{N}\right\},$$
and set
$$K_0:=\left\{B\in \go:B\subseteq A ~ \mathrm{for ~ some} ~ A\in Y\right\}.$$
We construct a specific saturated hereditary subset of $\go$ containing $H$ as follows: for any $n\in \mathbb{N}$ inductively define
\begin{multline*}
S_n:=\left\{w\in G^0: 0<|s_\g^{-1}(w)|<\infty ~~ \mathrm{and} ~~ r_\g(s_\g^{-1}(w))\subseteq K_{n-1}\right\}\\
\cup \left\{w\in B: w\notin K_{n-1} ~~\mathrm{and} ~~ r_\g(s_\g^{-1}(w))\subseteq K_{n-1}\right\}
\end{multline*}
and
$$K_n:=\left\{A\cup F: A\in K_{n-1} ~~ \mathrm{and} ~~ F\subseteq S_n ~~\mathrm{is ~ a ~ finite ~ subset} \right\}.$$
Then we can easily show that the set
$$K=\bigcup_{n=0}^\infty K_n=\left\{A\cup F: A\in K_0 ~~ \mathrm{and} ~~ F\subseteq \bigcup_{n=1}^\infty S_n ~~ \mathrm{is~a ~ finite ~ subset}\right\}$$
is hereditary and saturated in $\g$.

\begin{lem}\label{lem6.3}
Suppose that $\gh$ is a quotient ultragraph of $\g$, and $\gamma=e_1\ldots e_n$ is a loop in $\gh$ such that there are no loops $\alpha$ with $s(\alpha)=s(\gamma)$ and $\alpha\nsubseteq \gamma$, $\gamma\nsubseteq\alpha$. If we construct the set $K$ as above, then $K$ is a saturated hereditary subset of $\go$ such that $H\subseteq K$ and $B\subseteq K\cup B_K$. Moreover, we have $A\cap \gamma^0=\emptyset$ for every $A\in K$.
\end{lem}

\begin{proof}
First, we show inductively that every $K_n$ is a hereditary subset of $\go$ by checking the conditions of Definition \ref{defn2.5}. To verify condition (1) for $K_0$, let us take $e\in \gl$ with $s_\g(e)\in K_0$. Then $s_\g(e)\in H\cup X$. If $s_\g(e)\in H$, we have $r_\g(e)\in H\subseteq K_0$ by the hereditary property of $H$. If $s_\g(e)\in X$, there is $\alpha\in \g^*$ such that $s_\g(\alpha)\in \gamma^0$ and $s_\g(e)\in r_\g(\alpha)\setminus \gamma^0$. Hence, $s_\g(\alpha e)=s_\g(\alpha)\in \gamma^0$. Also, $r_\g(\alpha e)\cap \gamma^0=\emptyset$ because the otherwise implies the existence of a path $\beta\in \g^*$ with $s_\g(\beta)=s_\g(\gamma)$ and $\beta\nsubseteq \gamma$, $\gamma\nsubseteq \beta$ that contradicts the hypothesis. It turns out
$$r_\g(e)=r_\g(\alpha e)=r_\g(\alpha e)\setminus \gamma^0 \in X\subseteq K_0.$$
Hence, $K_0$ satisfies condition (1) in Definition \ref{defn2.5}. We may easily verify conditions (2) and (3) for $K_0$, so $K_0$ is hereditary. Moreover, for every $w\in S_n$, the range of each edge emitted by $w$ belongs to $K_{n-1}$ by definition. Thus, we can inductively check that each $K_n$ is hereditary, and so is $K=\cup_{n=1}^\infty K_n$. The saturation property of $K$ is similar to the proof of \cite[Lemma 3.12]{tom2}, and it is omitted.

Furthermore, it is clear that $H\subseteq K$ by the fact $H\subseteq Y$. To see $B\subseteq K\cup B_K$, take an arbitrary vertex $w\in B\setminus K$. Then $w$ is an infinite emitter and since $w\notin \cup_{n=1}^\infty S_n$, we have $r_\g(s_\g^{-1}(w))\nsubseteq \cup_{n=1}^\infty K_n=K$. Thus $w$ emits some edges into $\go\setminus K$ which implies $w\in B_K$ as desired.

It remains to show $A\cap \gamma^0=\emptyset$ for every $A\in K$. To do this, note that $A\cap \gamma^0=\emptyset$ for every $A\in K_0$ because this property holds for all $A\in H$ and $A\in X$. We claim that $(\cup_{n=1}^\infty S_n)\cap \gamma^0=\emptyset$. Indeed, if $v=s_\g(e_i)\in \gamma^0$ for some $e_i\in \gamma$, then $r_\g(e_i)\cap \gamma^0\neq \emptyset$ and $r_\g(e_i)\notin K_0$. Hence, $\{r_\g(e):e\in \gl,~s_\g(e)=v\}\nsubseteq K_0$ that turns out $v\notin S_1$. So, we have $S_1\cap \gamma^0=\emptyset$. An inductive argument shows $S_n\cap \gamma^0=\emptyset$ for $n\geq 1$, and the claim holds. Now since
$$K=\cup_{n=1}^\infty K_n=\left\{A\cup F: A\in K_0 ~~\mathrm{and} ~~ F\subseteq \cup_{n=1}^\infty S_n ~~ \mathrm{is~ a ~ finite ~ subset}\right\},$$
we conclude that $A\cap \gamma^0=\emptyset$ for all $A\in K$.
\end{proof}

\begin{prop}\label{prop6.4}
A quotient ultragraph $\gh$ satisfies Condition (K) if and only if for every admissible pair $(K,S)$ in $\g$ with $H\subseteq K$ and $B\subseteq K\cup S$, the quotient ultragraph $\g/(K,S)$ satisfies Condition (L).
\end{prop}

\begin{proof}
Suppose that $\gh$ satisfies Condition (K) and $(K,S)$ is an admissible pair in $\g$ with $H\subseteq K$. Let $\alpha=e_1\ldots e_n$ be a loop in $\g/(K,S)$. Since $\alpha$ is also a loop in $\gh$ and $\gh$ satisfies Condition (K), there is a loop $\beta=f_1\ldots f_m$ in $\g$ with $s_\g(\alpha)=s_\g(\beta)$, and neither $\alpha\subseteq \beta$ nor $\beta\subseteq \alpha$. Without loos of generality, assume $e_1\neq f_1$. By the fact $s_\g(\alpha)=s_\g(\beta)\in r_\g(\beta)$, we have $r_\g(\beta)\notin K$, and so $r_\g(f_1)\notin K$ by the hereditary property of $K$. Therefore, $f_1$ is an exit for $\alpha$ in $\g/(K,S)$ and we conclude that $\g/(K,S)$ satisfies Condition (L).

For the converse, suppose on the contrary that $\gh$ does not satisfy Condition (K). Then there exists a loop $\gamma=e_1\ldots e_n$ in $\gh$ such that there are no loops $\alpha$ with $s(\alpha)=s(\gamma)$, $\alpha\nsubseteq \gamma$, and $\gamma\nsubseteq \alpha$. As Lemma \ref{lem6.3}, construct a saturated hereditary subset $K$ of $\go$ and consider the quotient ultragraph $\g/(K,B_K)=(\Psi(G^0),\Psi(\go),\Psi(\gl),r',s')$. We denote by $\lfloor . \rfloor$ the equivalent classes in $\Psi(\go)$. We show that $\gamma$ as a loop in $\g/(K,B_K)$ has no exits and $r'(e_i)=s'(e_{i+1})$ for $1\leq i\leq n$. If $f$ is an exit for $\gamma$ in $\g/(K,B_K)$ such that $s'(f)=s'(e_j)$ and $f\neq e_j$, then $r_\g(f)\notin K$ and $r_\g(f)\cap \gamma^0\neq \emptyset$ (if $r_\g(f)\cap \gamma^0= \emptyset$, then $r_\g(f)=r_\g(f)\setminus \gamma^0\in X\subseteq K$, a contradiction). So, there is $e_l\in \gamma$ such that $s_\g(e_l)\in r_\g(f)$. If we set $\alpha:=e_1\ldots e_{j-1}fe_l\ldots e_n$, then $\alpha$ is a loop in $\g$ with $s_\g(\alpha)=s_\g(\gamma)$, and $\alpha\nsubseteq \gamma$, $\gamma\nsubseteq \alpha$, that contradicts the hypothesis. Therefore, $\gamma$ has no exits in $\g/(K,B_K)$. Moreover, we have $r'(e_i)\cap \lfloor\gamma^0\rfloor=s'(e_{i+1})$ for each $1\leq i\leq n$, because the otherwise gives an exit for $\gamma$ in $\g/(K,B_K)$ by the construction of $K$. Hence,
$$r'(e_i)\setminus s'(e_{i+1})= r'(e_i)\setminus \lfloor\gamma^0\rfloor=\lfloor\emptyset\rfloor$$
and we get $r'(e_i)=s'(e_{i+1})$ (note that the fact $r_\g(e_i)\setminus \gamma^0\in K$ implies $r'(e_i)\setminus \lfloor\gamma^0 \rfloor=\lfloor \overline{r_\g(e_i)}\setminus \gamma^0 \rfloor=\lfloor\emptyset \rfloor$). Therefore, the quotient ultragraph $\g/(K,B_K)$ does not satisfy Condition (L) as desired.
\end{proof}

To prove the main result of this section, Theorem \ref{thm6.6}, we need also the following lemma.

\begin{lem}\label{lem6.5}
Let $\g/(H,B)=(\pGo,\pgo, \pgl,r,s)$ be a quotient ultragraph of $\g$. If $\gh$ does not satisfy Condition (L), then $C^*(\g/(H,B))$ contains an ideal Morita-equivalent to $C(\mathbb{T})$. In particular, $C^*(\gh)$ contains non-gauge invariant ideals.
\end{lem}

\begin{proof}
Suppose that $\gamma=e_1\ldots e_n$ is a loop in $\g/(H,B)$ without exits and $r(e_i)=s(e_{i+1})$ for $1\leq i\leq n$. If $C^*(\gh)=C^*(t_e,q_{[A]})$, for each $i$ we have
$$t_{e_i}^*t_{e_i}=q_{r(e_i)}=q_{s(e_{i+1})}=t_{e_{i+1}}t_{e_{i+1}}^*.$$
Write $[v]:=s(\gamma)$ and let $I_{\gamma}$ be the ideal of $C^*(\g/(H,B))$ generated by $q_{[v]}$. Since $\gamma$ has no exits in $\gh$ and we have
$$q_{s(e_i)}=(t_{e_i}\ldots t_{e_n})q_{[v]}(t_{e_n}^* \ldots t_{e_i}^*)$$
for every $1\leq i\leq n$, an easy argument shows that
$$I_{\gamma}=\overline{\mathrm{span}}\left\{t_{\alpha}q_{[v]}t_{\beta}^*: \alpha,\beta\in (\g/(H,B))^*, [v]\subseteq r(\alpha)\cap r(\beta)\right\}.$$
So, we have
$$q_{[v]}I_{\gamma}q_{[v]}=\overline{\mathrm{span}}\left\{(t_{\gamma})^n q_{[v]} (t_{\gamma}^*)^m: m,n\geq 0\right\},$$
where $(t_\gamma)^0=(t_\gamma^*)^0:=q_{[v]}$. We show that $q_{[v]}I_{\gamma}q_{[v]}$ is a full corner in $I_{\gamma}$ which is isometrically isomorphic to $C(\mathbb{T})$. For this, associated to the graph $E$
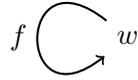
\[\begin{tikzpicture}[shorten >=1pt,auto,thick]
\node (w) at (-2,0) {$w$};
\draw[->] (w) .. controls (-3.5,1.2) and (-3.5,-1.2) .. node[left] {$f$} (w);
\end{tikzpicture}
\]
define $Q_w:=q_{[v]}$ and $T_f:=t_{\gamma}~(=t_{\gamma}q_{[v]})$. Then $\{T_f,Q_w\}$ is a Cuntz-Krieger $E$-family in $q_{[v]}I_{\gamma}q_{[v]}$. Assume $C^*(E)=C^*(s_f,p_w)$. Since $Q_w\neq 0$, the gauge-invariant uniqueness theorem for graph $C^*$-algebras implies that the $*$-homomorphism $\phi:C^*(E)\rightarrow q_{[v]}I_{\gamma}q_{[v]}$ with $p_w\mapsto Q_w$ and $s_f\mapsto T_f$ is faithful. Moreover, the $C^*$-algebra $q_{[v]}I_{\gamma}q_{[v]}$ is generated by $\{T_f,Q_w\}$, and hence, $\phi$ is an isomorphism. We know that $C^*(E)\cong C(\mathbb{T})$, so $q_{[v]}I_{\gamma}q_{[v]}$ is isomorphic to $C(\mathbb{T})$. The fullness of $q_{[v]}I_{\gamma}q_{[v]}$ in $I_{\gamma}$ may be easily shown. Indeed, if $J$ is an ideal in $I_{\gamma}$ with $q_{[v]}I_{\gamma}q_{[v]}\subseteq J$, we then have $q_{[v]}\in J$ and $J=I_{\gamma}$. Therefore, $I_\gamma$ is Morita-equivalent to $q_{[v]}I_\gamma q_{[v]}\cong C(\mathbb{T})$ as desired.

On the other hand,  $C(\mathbb{T})$ contains infinitely many non-gauge invariant ideals (corresponding with closed subsets $\emptyset \neq U\subsetneq\mathbb{T}$). Therefore,  the corner $q_{[v]} I_{\gamma} q_{[v]}$ contains non-gauge invariant ideals, and so does $I_{\gamma}$ by the Morita-equivalence. This follows the second statement of lemma because each ideal of $I_{\gamma}$ is an ideal of $C^*(\g/(H,B))$.
\end{proof}

We now ready to prove the main result of this section as a generalization of \cite[Corollary 3.8]{bat} and \cite[Proposition 7.3]{kat3}.

\begin{thm}\label{thm6.6}
A quotient ultragraph $\gh$ satisfies Condition (K) if and only if all ideals of $C^*(\gh)$ are gauge invariant.
\end{thm}

\begin{proof}
Suppose that $\gh$ satisfies Condition (K). Take an arbitrary ideal $J$ in $C^*(\gh)$ and let $I$ be its corresponding ideal in $C^*(\g)$ with $I_{(H,B)}\subseteq I$ (by Proposition \ref{prop5.1}). If $C^*(\g)=C^*(s_e,p_A)$ and set
$$K:=\left\{A\in \g^0:p_A\in I\right\} ~~,~~ S:=\left\{w\in B_K:p_w-\sum_{s(e)=w,r(e)\notin K}s_e s_e^*\in I\right\},$$
then \cite[Lemma 3.4]{tom2} implies that $(K,S)$ is an admissible pair in $\g$. Since $I_{(K,S)}\subseteq I$, the map
$$\left\{
    \begin{array}{ll}
       \phi:C^*(\g)/{I_{(K,S)}}\longrightarrow {C^*(\g)}/{I} &  \\
       \hspace{1.1cm}a+I_{(K,S)}\longmapsto a+I & (\mathrm{for} ~~ a\in C^*(\g) )
    \end{array}
  \right.
$$
is a well-defined $*$-homomorphism. Let us denote $\pi:C^*(\g/(K,S))\rightarrow C^*(\g)/I_{(K,S)}$ the isomorphism of Proposition \ref{prop5.1}. Since the quotient ultragraph $\g/(K,S)$ satisfies Condition (L) by Proposition \ref{prop6.4}, the Cuntz-Krieger uniqueness theorem, Theorem \ref{thm4.8}, implies that $\phi\circ \pi$ is injective. So, $\phi$ is injective that follows $I=I_{(K,S)}$ and $I$ is gauge invariant. Now as $J$ is the image of $I_{(K,S)}$ into the quotient $C^*(\g)/I_{(H,B)}\cong C^*(\gh)$, we conclude that $J=J_{[K,S]}$ is a gauge invariant ideal in $C^*(\gh)$.

Conversely, assume $\gh$ does not satisfy Condition (K). By Proposition \ref{prop6.4}, there exists an admissible pair $(K,S)$ in $\g$ with $H\subseteq K$ and $B\subseteq K\cup S$ such that the quotient ultragraph $\g/(K,S)$ does not satisfy Condition (L). Note that in this case, $I_{(H,B)}\subseteq I_{(K,S)}$ and if $J_{[K,S]}$ is the image of $I_{(K,S)}$ into $C^*(\g)/I_{(H,B)}\cong C^*(\gh)$, then $C^*(\gh)/J_{[K,S]}\cong C^*(\g/(K,S))$. But Lemma \ref{lem6.5} says that $C^*(\g/(K,S))$ contains infinitely many  non-gauge invariant ideals. On the other hand, for every gauge invariant ideal $J$ of $C^*(\gh)$, the ideal $J+J_{[K,S]}$ is also gauge invariant in $C^*(\gh)/J_{[K,S]}\cong C^*(\g/(K,S))$. This concludes that the $C^*$-algebra $C^*(\gh)$ contains non-gauge invariant ideals and completes the proof.
\end{proof}

We gather the results of this section in the following corollary. Recall from \cite{bro} that the \emph{real rank} of a unital $C^*$-algebra $A$ is the smallest integer $\mathrm{RR}(A)$ such that for every $\varepsilon>0$, positive integer $n\leq \mathrm{RR}(A)+1$ and  $n$-tuple $(x_1,\ldots, x_n)$ of self-adjoint elements in $A$, there is an $n$-tuple $(y_1,\ldots,y_n)$ of self-adjoint elements of $A$ such that $\sum_{i=1}^n y_i^2$ is invertible and
$$\left\|\sum_{i=1}^n(x_i-y_i)^2\right\|<\varepsilon.$$
If $A$ is non-unital, $\mathrm{RR}(A)$ is the real rank of its unitization.

\begin{cor}\label{cor6.7}
Let $H$ be a saturated hereditary subset of $\go$. Then the following conditions are equivalent:
\begin{enumerate}
  \item The quotient ultragraph $\g/(H,B_H)$ satisfies Condition (K).
  \item The quotient ultragraph $\gh$ satisfies Condition (K) for some $B\subseteq B_H$.
  \item For every admissible pair $(K,S)$ with $H\subseteq K$, $\g/(K,S)$ satisfies Condition (L).
  \item If $B\subseteq B_H$, all ideals of $C^*(\gh)\cong C^*(\g)/I_{(H,B)}$ are gauge invariant.
  \item If $B\subseteq B_H$, the real rank of $C^*(\gh)\cong C^*(\g)/I_{(H,B)}$ is zero.
\end{enumerate}
\end{cor}

\begin{proof}
We have shown the equivalence of conditions (1)-(4). So, it suffices to show the implications (2) $\Rightarrow$ (5) and (5) $\Rightarrow$ (3). For (2) $\Rightarrow$ (5), suppose $B\subseteq B_H$ and the quotient ultragraph $\gh$ satisfies Condition (K). Select an increasing sequence $\{F_n\}_{n=1}^\infty$ of finite subsets of $\sg \cup \pgl$ such that $\cup_{n=1}^\infty F_n=\sg \cup \pgl$. Then, similar to the proof of Theorem \ref{thm4.5}, $C^*(\gh)$ is isomorphic to the inductive limit $\underrightarrow{\lim} C^*(G_{F_n})$. Since loops in each $G_{F_n}$ come from those of $\gh$, we may select $F_n$'s such that each finite graph $G_{F_n}$  satisfies Condition (K), and so, the real rank of $C^*(G_{F_n})$ is zero \cite[Theorem 4.1]{jeo2}. Thus the real rank of $C^*(\gh)$ is zero by \cite[Proposition 3.1]{bro}.

For (5) $\Rightarrow$ (3), suppose that there is an admissible pair $(K,S)$ with $H\subseteq K$ such that the quotient ultragraph $\g/(K,S)$ does not satisfy Condition (L). We can also assume that $B\subseteq K\cup S$ by Lemma \ref{lem6.3}. Then, by Lemma \ref{lem6.5}, $C^*(\g/(K,S))$ contains an ideal Morita-equivalent to $C(\mathbb{T})$, and so $\mathrm{RR}(C^*(\g/(K,S)))\neq 0$ by \cite[Corolary 2.8]{bro}. As $C^*(\g/(K,S))$ is a quotient of $C^*(\gh)$, Theorem 3.14 of \cite{bro} implies that $\mathrm{RR}(C^*(\gh))\neq 0$.
\end{proof}

For ultragraph $C^*$-algebras, Corollary \ref{cor6.7} shows both \cite[Proposition 7.3]{kat3} and \cite[Proposition 5.26]{kat2} because every ultragraph can be considered as a quotient ultragraph with the trivial admissible pair $(\{\emptyset\}, \{\emptyset\})$. However, our proof is quite different from those of \cite{kat3} and \cite{kat2}.

\begin{cor}\label{cor6.8}
An ultragraph $\g$ satisfies Condition (K) if and only if all ideals of $C^*(\g)$ are gauge invariant if and only if the real rank of $C^*(\g)$ is zero.
\end{cor}


\section{Primitive ideals in $C^*(\g)$}

In this section, we apply quotient ultragraphs to describe primitive gauge invariant ideals of an ultragraph $C^*$-algebra. Recall that since every ultragraph $C^*$-algebra $C^*(\g)$ is separable (as assumed $\go$ to be countable), a prime ideal of $C^*(\g)$ is primitive and vice versa \cite[Corollaire 1]{dix}.

\begin{defn}
Let $\g$ be an ultragraph. For two sets $A,B\in \go$, we write $A\geq B$ if either $B\subseteq A$, or there exists $\alpha\in \g^*$ with $|\alpha|\geq 1$ such that $s(\alpha)\in A$ and $B\subseteq r(\alpha)$. We simply write $A\geq v$, $v\geq B$, and $v\geq w$ if $A\geq \{v\}$, $\{v\}\geq B$, and $\{v\}\geq \{w\}$, respectively. A subset $M\subseteq \go$ is said to be \emph{downward directed} whenever for every $A,B\in M$, there exists $\emptyset\neq C\in M$ such that $A,B\geq C$.
\end{defn}

To prove Proposition \ref{prop7.3} below, we need a simple lemma.

\begin{lem}\label{lem7.2}
If $\gh$ satisfies Condition (L), then every nonzero ideal of $C^*(\gh)$ contains some projection $q_{[A]}$ with $[A]\neq [\emptyset]$.
\end{lem}

\begin{proof}
Take an arbitrary ideal $J$ in $C^*(\gh)$. If there are no $q_{[A]}\in J$ with $[A]\neq[\emptyset]$, then the Cuntz-Krieger uniqueness theorem implies that the quotient homomorphism $\phi:C^*(\gh)\rightarrow C^*(\gh)/J$ is injective. Hence, we have $J=\ker \phi=(0)$.
\end{proof}

\begin{prop}\label{prop7.3}
Let $H$ be a saturated hereditary subset of $\go$. Then the ideal $I_{(H,B_H)}$ in $C^*(\g)$ is primitive if and only if the quotient ultragraph $\g/(H,B_H)$ satisfies Condition (L) and the collection $\go \setminus H$ is downward directed.
\end{prop}

\begin{proof}
Let $I_{(H,B_H)}$ be a primitive ideal of $C^*(\g)$. Since $C^*(\g)/I_{(H,B_H)}\cong C^*(\g/(H,B_H))$, the zero ideal in $C^*(\g/(H,B_H))$ is primitive. If $\g/(H,B_H)$ does not satisfy Condition (L), then $C^*(\g/(H,B_H))$ contains an ideal $J$ Morita-equivalent to $C(\mathbb{T})$  by Lemma \ref{lem6.5}. Select two ideals $I_1,I_2$ in $C(\mathbb{T})$ with $I_1\cap I_2=(0)$, and let $J_1,J_2$ be their corresponding ideals in $J$. Then $J_1$ and $J_2$ are two nonzero ideals of $C^*(\g/(H,B_H))$ with $J_1\cap J_2=(0)$, contradicting the primness of $C^*(\g/(H,B_H))$. Therefore, $\gh$ satisfies Condition (L).

Now we show that $M:=\go\setminus H$ is downward directed. For this, we take arbitrary sets $A,B\in M$ and consider the ideals
$$J_1:=C^*(\g/(H,B_H)) q_{[A]} C^*(\g/(H,B_H))$$
and
$$J_2:=C^*(\g/(H,B_H)) q_{[B]} C^*(\g/(H,B_H))$$
in $C^*(\g/(H,B_H))$ generated by $q_{[A]}$ and $q_{[B]}$, respectively. Since $A,B\notin H$, the projections $q_{[A]},q_{[B]}$ are nonzero by Theorem \ref{thm3.11}, and so are the ideals $J_1,J_2$. The primness of $C^*(\g/(H,B_H))$ implies that the ideal
$$J_1J_2=C^*\left(\g/(H,B_H)\right) q_{[A]} C^*\left(\g/(H,B_H)\right) q_{[B]} C^*\left(\g/(H,B_H)\right)$$
is nonzero, and hence $q_{[A]} C^*(\g/(H,B_H)) q_{[B]} \neq \{0\}$. As the set
$$\mathrm{span}\left\{ t_\alpha q_{[D]} t_\beta^* : \alpha, \beta \in (\gh)^*,~ r(\alpha)\cap [D]\cap r(\beta)\neq [\emptyset] \right\}$$
is dense in $C^*(\g/(H,B_H))$, there exist $\alpha,\beta\in (\g/(H,B_H))^*$ and $[D]\in \pgo$ such that $q_{[A]}(t_\alpha q_{[D]} t_\beta^*) q_{[B]}\neq 0$. In this case, we must have $s(\alpha)\subseteq [A]$ and $s(\beta)\subseteq [B]$ and thus, $A,B\geq C$ for $C:=r_\g(\alpha) \cap D \cap r_\g(\beta)$.

In order to prove the converse, assume that $\g/(H,B_H)$ satisfies Condition (L) and the collection $M=\go\setminus H$ is downward directed. Fix two nonzero ideals $J_1,J_2$ of $C^*(\g/(H,B_H))$. By Lemma \ref{lem7.2}, there are nonzero projections $q_{[A]}\in J_1$ and $q_{[B]}\in J_2$. Then $A,B\notin H$ and, since $M$ is downward directed, there exists $C\in M$ such that $A,B\geq C$. Hence, the ideal $J_1\cap J_2$ contains the nonzero projection $q_{[C]}$. Since $J_1$ and $J_2$ were arbitrary, this implies that the $C^*$-algebra $C^*(\g/(H,B_H))$ is primitive and consequently, $I_{(H,B_H)}$ is a primitive ideal in $C^*(\g)$ by Proposition \ref{prop5.1}.
\end{proof}

The next proposition describes another kind of primitive ideals in $C^*(\g)$.

\begin{prop}\label{prop7.4}
Let $(H,B)$ be an admissible pair in $\g$ and let $B=B_H\setminus \{w\}$. Then the ideal $I_{(H,B)}$ in $C^*(\g)$ is primitive if and only if $A\geq w$ for all $A\in \go\setminus H$.
\end{prop}

\begin{proof}
Suppose that $I_{(H,B)}$ is a primitive ideal and take an arbitrary $A\in \go\setminus H$. If $\overline{A}:=A\cup \{v':v\in A\cap (B_H\setminus B)\}$, then $q_{[\overline{A}]}$ and $q_{[w']}$ are two nonzero projections in $C^*(\g/(H,B))$. Consider two ideals
$$J_{[\overline{A}]}:=C^*(\gh)q_{[\overline{A}]} C^*(\gh)$$
and
$$J_{[w']}:= C^*(\gh) q_{[w']} C^*(\gh)$$
of $C^*(\gh)$ generated by $q_{[\overline{A}]}$ and $q_{[w']}$, respectively. The primness of $C^*(\gh)\cong C^*(\g)/I_{H,B}$ implies that the ideal
$$J_{[\overline{A}]} J_{[w']}=C^*(\g/(H,B)) q_{[\overline{A}]} C^*(\g/(H,B)) q_{[w']} C^*(\g/(H,B))$$
is nonzero, and hence $q_{[\overline{A}]} C^*(\g/(H,B)) q_{[w']}\neq \{0\}$. Then there exist $\alpha,\beta\in (\g/(H,B))^*$ such that $q_{[\overline{A}]}t_\alpha t_\beta^* q_{[w']}\neq 0$. Since $[w']$ is a sink in $\g/(H,B)$, we must have $q_{[\overline{A}]} t_\alpha q_{[w']}\neq 0$. If $|\alpha|=0$, then $[w']\subseteq [\overline{A}]$, $w'\in \overline{A}$ and $w\in A$. If $|\alpha|\geq 1$, then $s(\alpha)\subseteq [\overline{A}]$ and $[w']\subseteq r(\alpha)$, which yield $s_\g(\alpha)\in A$ and $w\in r_\g(\alpha)$. Therefore, we obtain $A\geq w$ in either case.

Conversely, assume $A\geq w$ for every $A\in \go\setminus H$. Then the collection $\go\setminus H$ is downward directed. Moreover, this hypothesis implies that, for every $[\emptyset]\ne [A]\in \pgo$, there exists $\alpha\in (\gh)^*$ such that $s(\alpha)\subseteq [A]$ and $[w']\subseteq r(\alpha)$. Since $[w']$ is a sink in $\gh$, we see that the quotient ultragraph $\gh$ satisfies Condition (L). Now similar to the last part of the proof of Proposition \ref{prop7.3}, we can show that $I_{(H,B)}$ is a primitive ideal.
\end{proof}

Recall that the loops in a quotient ultragraph $\gh$ come from those in the initial ultragraph $\g$. So, to check Condition (L) for a quotient ultragraph $\gh$, we can use the following definition.

\begin{defn}
Let $H$ be a saturated hereditary subset of $\go$ and denote $M:=\go\setminus H$. A loop $\alpha=e_1\ldots e_n$ is said to be in $\g\setminus H$ if $r_\g(\alpha)\in M$. In this case, we say that $\alpha$ has \emph{an exit in $\g\setminus H$} if either $r_\g(e_i)\setminus s_\g(e_{i+1})\in M$ for some $i$, or there is an edge $f$ with $r_\g(f)\in M$ such that $s_\g(f)=s_\g(e_i)$ and $f\neq e_i$, for some $1\leq i\leq n$.
\end{defn}

It is easy to verify that a quotient ultragraph $\g/(H,B)$ satisfies Condition (L) if and only if every loop in $\g\setminus H$ has an exit in $\g\setminus H$. Now we characterize all primitive gauge invariant ideals of an ultragraph $C^*$-algebra $C^*(\g)$ that is a generalization of \cite[Theorem 4.7]{bat} and \cite[Theorem 4.5]{dri}.

\begin{thm}\label{thm7.6}
Let $\g$ be an ultragraph. A gauge invariant ideal $I_{(H,B)}$ of $C^*(\g)$ is primitive if and only if one of the following holds:
\begin{enumerate}
\item $B=B_H$, $\go\setminus H$ is downward directed, and every loop in $\g\setminus H$ has exits in $\g\setminus H$.
\item $B=B_H\setminus\{w\}$ for some $w\in B_H$, and $A\geq w$ for all $A\in \go\setminus H$.
\end{enumerate}
\end{thm}

\begin{proof}
Let $I_{(H,B)}$ be a primitive ideal in $C^*(\g)$. Then $C^*(\gh)\cong C^*(\g)/I_{(H,B)}$ is a primitive $C^*$-algebra. We claim that $|B_H\setminus B|\leq 1$. Indeed, if $w_1,w_2$ are two distinct vertices in $B_H\setminus B$, similar to the proof of Propositions \ref{prop7.3} and \ref{prop7.4}, the primitivity of $C^*(\gh)$ implies that the corner $q_{[w_1']} C^*(\gh) q_{[w_2']}$ is nonzero. So, there exist $\alpha,\beta\in (\gh)^*$ such that $q_{[w_1']}t_\alpha t_\beta^* q_{[w_2']}\neq 0$. But we must have $|\alpha|=|\beta|=0$ because $[w'_1],[w'_2]$ are two sinks in $\gh$.  Hence, $q_{[w_1']}q_{[w_2']}\neq 0$ which is impossible because $q_{[w_1']}q_{[w_2']}=q_{[\{w_1'\}\cap \{w_2'\}]}=q_{[\emptyset]}=0$. Thus, the claim holds. Now we may apply Propositions \ref{prop7.3} and \ref{prop7.4} to obtain the result.
\end{proof}

Corollary \ref{cor6.8} says that if $\g$ satisfies Condition (K), then all ideals of $C^*(\g)$ are of the form $I_{(H,B)}$. So, we have the following.

\begin{cor}
If an ultragraph $\g$ satisfies Condition (K), then Theorem \ref{thm7.6} describes all primitive ideals of $C^*(\g)$.
\end{cor}

In the end of paper, we establish primitive gauge invariant ideals of a quotient ultragraph $C^*$-algebra $C^*(\gh)$ by applying Theorem \ref{thm7.6}. Recall from Theorem \ref{thm5.3} that every gauge invariant ideal of $C^*(\gh)$ is of the form $J_{[K,S]}$ with $H\subseteq K$ and $B\subseteq K\cup S$. Since $$\frac{C^*(\gh)}{J_{[K,S]}}\cong C^*(\g/(K,S)) \cong \frac{C^*(\g)}{I_{(K,S)}}$$
by Theorem \ref{thm5.3}, a gauge invariant ideal $J_{[K,S]}$ of $C^*(\gh)$ is primitive if and only if $I_{(K,S)}$ is a primitive ideal of $C^*(\g)$. Therefore, Theorem \ref{thm7.6} conclude that:

\begin{thm}
Let $\gh$ be a quotient ultragraph of $\g$. A gauge invariant ideal $J_{[K,S]}$ of $C^*(\gh)$ is primitive if and only if one of the following conditions holds:
\begin{enumerate}
\item $S=B_K$, $\go\setminus K$ is downward directed, and every loop in $\g\setminus K$ has an exit in $\g\setminus K$.
\item $S=B_K\setminus\{w\}$ for some $w\in B_K$, and $A\geq w$ for all $A\in \go\setminus K$.
\end{enumerate}
In particular, if $\gh$ satisfies Condition (K), these conditions characterize all primitive ideals of $C^*(\gh)$.
\end{thm}



\begin{thebibliography}{99}


\bibitem{bat}T. Bates, J.H. Hong, I. Raeburn and W. Szyma$\acute{\mathrm{n}}$ski, \emph{The ideal structure of the $C^*$-algebras of infinite graphs}, Ilinois J. Math. {\bf46} (2002), 1159-1176.

\bibitem{bat2} T. Bates, D. Pask, I. Raeburn and W. Szyma$\acute{\mathrm{n}}$ski, \emph{The $C^*$-algebras of row-finite graphs}, New York J. Math. {\bf 6} (2000), 307-324.

\bibitem{bro} L.G. Brown and G.K. Pedersen, \emph{C*-algebras of real rank zero}, J. Funct. Anal. {\bf99} (1991), 131-149.

\bibitem{dix} J. Dixmier, \emph{Sur les $C^*$-algebres}, Bull. Soc. Math. France {\bf88} (1960), 95-112.

\bibitem{dri}D. Drinen and M. Tomforde, \emph{The $C^*$-algebras of arbitrary graphs}, Rocky Mt. J. Math. {\bf35} (2005), 105-135.

\bibitem{exe}R. Exel and M. Laca, \emph{Cuntz-Krieger algebras for infinite matrices}, J. Reine Angew. Math. {\bf512} (1999), 119-172.

\bibitem{fow}N. Fowler, M. Laca and I. Raeburn, \emph{The $C^*$-algebras of infinite graphs}, Proc. Amer. Math. Soc. {\bf8} (2000), 2319-2327.

\bibitem{hon} J.H. Hong and W. Szyma$\acute{\mathrm{n}}$ski, \emph{Purely infinite Cuntz-Krieger algebras of directed graphs}, Bull. London Math. Soc. {\bf 35} (2003), 689-696.

\bibitem{jeo} J.A. Jeong, S.H. Him and G.H. Park, \emph{The structure of gauge-invariant ideals of labelled graph $C^*$-algebras}, J. Funct. Anal. {\bf262} (2012), 1759-1780.

\bibitem{jeo2} J.A. Jeong, G.H. Park, \emph{Graph $C^*$-algebras with real rank zero}, J. Funct. Anal. {\bf188} (2002), 216-226.

\bibitem{kat1} T. Katsura, \emph{A class of $C^*$-algebras generalizing both graph algebras and homeomorphism $C^*$-algebras, I. Fundamental
results}, Trans. Amer. Math. Soc. {\bf356} (2004), 4287-4322.

\bibitem{kat2} T. Katsura, P.S. Muhly, A. Sims and M. Tomforde, \emph{Graph algebras, Exel-Laca algebras, and ultragraph $C^*$-algebras coincide up to Morita equivalence}, J. Reine Angew. Math. {\bf640} (2010), 135-165.

\bibitem{kat3} T. Katsura, P.S. Muhly, A. Sims and M. Tomforde, \emph{Utragraph algebras via topological quivers}, Studia Math. {\bf187} (2008), 137-155.

\bibitem{muh} P.S. Muhly and M. Tomforde, \emph{Topological Quivers}, Internat. J. Math. {\bf16} (2005), 693-755.

\bibitem{rae} I. Raeburn and W. Szyma´nski, \emph{Cuntz-Krieger algebras of infinite graphs and matrices}, Trans. Amer. Math. Soc. {\bf356} (2004), 39-59.

\bibitem{szy}W. Szyma$\acute{\mathrm{n}}$ski, \emph{Simpliciy of Cuntz-Krieger algebras of infinite matrices}, Pacific J. Math. {\bf199} (2001), 249-256.

\bibitem{tom}M. Tomforde, \emph{A unified approach to Exel-Laca algebras and $C^*$-algebras associated to graphs}, J. Oper. Theory {\bf50} (2003), 345-368.

\bibitem{tom2}M. Tomforde, \emph{Simplicity of ultragraph algebras}, Indiana Univ. Math. J. {\bf52} (2003) 901-925.

\end{thebibliography}
\end{document}